\documentclass[11pt]{article}

\usepackage[margin=1in]{geometry}

\usepackage{amsmath}
\usepackage{amssymb}
\usepackage{amsthm}
\usepackage{mathtools}
\usepackage{stmaryrd}
\usepackage{bm}
\usepackage{booktabs}
\usepackage{multirow}
\usepackage{tikz}
\usepackage{graphicx}
\usepackage[hidelinks]{hyperref}
\usepackage[capitalize,nameinlink]{cleveref}

\newcommand{\vect}[1]{\bm{#1}}
\newcommand{\matr}[1]{\bm{#1}}
\newcommand{\tens}[1]{\bm{#1}}

\newcommand{\grad}{\nabla}
\DeclareMathOperator{\divg}{div}
\DeclareMathOperator{\curl}{curl}

\newcommand{\symgrad}{\nabla^{\mathrm{s}}}

\newcommand{\dif}{\,\mathrm{d}}

\DeclareMathOperator{\tr}{tr}
\DeclareMathOperator{\dev}{dev}

\newcommand{\trans}{\mathsf{T}}

\newcommand{\jump}[1]{\llbracket #1 \rrbracket}

\newcommand{\defeq}{\triangleq}
\newcommand{\dvdots}{
  \mathbin{\vcenter{\offinterlineskip
    \hbox{$\cdot$}\kern-0.2ex
    \hbox{$\cdot$}}}
}
\newcommand{\tvdots}{
  \mathbin{\vcenter{\offinterlineskip
    \hbox{$\cdot$}\kern-0.2ex
    \hbox{$\cdot$}\kern-0.2ex
    \hbox{$\cdot$}}}
}

\newcommand{\tnorm}[1]{{\left\vert\kern-0.25ex\left\vert\kern-0.25ex\left\vert #1 \right\vert\kern-0.25ex\right\vert\kern-0.25ex\right\vert}}

\theoremstyle{plain}
\newtheorem{theorem}{Theorem}[section]
\newtheorem{lemma}[theorem]{Lemma}
\newtheorem{corollary}[theorem]{Corollary}
\newtheorem{proposition}[theorem]{Proposition}

\newtheorem{remark}[theorem]{Remark}
\newtheorem{hypothesis}[theorem]{Hypothesis}
\crefname{hypothesis}{Hypothesis}{Hypotheses}

\numberwithin{equation}{section}

\newcommand{\email}[1]{\texttt{#1}}
\newcommand{\funding}[1]{\noindent Funding: #1}

\newenvironment{keywords}
  {\par\smallskip\noindent\textbf{Key words.}\enspace}
  {\par}
\newenvironment{MSCcodes}
  {\par\smallskip\noindent\textbf{AMS subject classifications.}\enspace}
  {\par}

\title{A Hybridized Staggered Discontinuous Galerkin--Mixed Finite Element Method for Strain Gradient Elasticity
\thanks{\funding{The research of Eric Chung is partially supported by the Hong Kong RGC General Research Fund (Projects: 14305423 and 14305624), as well as the 1+1+1 CUHK-CUHK(SZ)-GDSTC Joint Collaboration Fund (Project: 2025A0505000059).}}}

\author{Bohan Yang\thanks{Department of Mathematics, The Chinese University of Hong Kong, Hong Kong Special Administrative Region
  (\email{bhyang@math.cuhk.edu.hk}).}
\and Eric T. Chung\thanks{Department of Mathematics, The Chinese University of Hong Kong, Hong Kong Special Administrative Region.
  (\email{tschung@math.cuhk.edu.hk}). Corresponding author.}}

\date{}

\hypersetup{
  pdftitle={A Hybridized Staggered Discontinuous Galerkin--Mixed Finite Element Method for Strain Gradient Elasticity},
  pdfauthor={Bohan Yang and Eric T. Chung}
}

\begin{document}

\maketitle

\begin{abstract}
    We propose a hybridized staggered discontinuous Galerkin--mixed finite element method for the strain gradient elasticity model, a fourth-order singularly perturbed problem governed by a material length scale parameter $\iota$ and the Lam\'e constants $\lambda$ and $\mu$. Introducing the total stress together with the scaled Cauchy and hyper stresses as auxiliary unknowns, we recast the model as a first-order system and discretize the displacement and the total stress by staggered discontinuous Galerkin spaces, the two scaled stresses by Raviart--Thomas pairs, with symmetry imposed strongly on all three stresses. The higher-order Dirichlet boundary condition enters the variational formulation naturally, so that the scheme avoids the numerical boundary layer that limits displacement-based methods. By establishing an inf-sup condition on the symmetric subspace and a discrete Korn inequality, we prove algebraic stability and optimal convergence, both uniform in $\iota$ and $\lambda$. We further develop a hybridized scheme in which local static condensation leaves only two multipliers on the mesh skeleton as globally coupled unknowns. Numerical experiments confirm the predicted convergence rates and the parameter robustness.
\end{abstract}

\begin{keywords}
    strain gradient elasticity, staggered discontinuous Galerkin method, mixed finite element method, hybridization, parameter-robustness, singular perturbation
\end{keywords}

\begin{MSCcodes}
    65N12, 65N15, 65N22, 65N30, 74S05
\end{MSCcodes}

\section{Introduction} \label{sec:1}

\subsection{Finite Element Methods for Strain Gradient Elasticity}

Classical elasticity is a second-order, scale-free continuum theory: it contains no intrinsic length scale and therefore cannot capture the pronounced size effects observed in micro- and nano-scale structures, such as the torsion of thin metallic wires \cite{fleck1994strain}, the bending of microbeams \cite{stolken1998microbend}, and nanoindentation \cite{nix1998indentation}. Strain gradient elasticity (SGE) accounts for these phenomena by incorporating the influence of strain gradients on the material response through one or more additional material length scale parameters; see \cite{askes2011gradient, polizzotto2015unifying} and the references therein.

This paper considers the SGE model proposed by Aifantis and his collaborators \cite{aifantis1992role, altan1992structure}, which retains a single length scale parameter $\iota \in (0, 1]$ while still capturing the essential size effect. Let $\Omega \subset \mathbb{R}^2$ be a bounded convex polygonal domain. In terms of the displacement field $\vect{u}$, the SGE equation reads
\begin{align}
    - (1 - \iota^2 \Delta) \left( \mu \Delta \vect{u} + (\lambda + \mu) \grad \divg \vect{u} \right) &= \vect{f} & & \text{in } \Omega, \label{eq:SGE} \\
    \vect{u} = \partial_{\vect{n}} \vect{u} &= \vect{0} & & \text{on } \partial \Omega, \label{eq:homogeneous-BC}
\end{align}
where $\lambda, \mu > 0$ are the Lam\'e constants, $\vect{f}$ is the body force, and $\vect{n}$ is the outward unit normal vector on $\partial \Omega$. We consider throughout the homogeneous doubly clamped boundary condition \eqref{eq:homogeneous-BC}, in which $\partial_{\vect{n}} \vect{u} \big|_{\partial \Omega} = \vect{0}$ is called the higher-order Dirichlet boundary condition.

Since \eqref{eq:SGE} is of fourth order, the first challenge in the discretization is the $H^2$-conformity requirement. The $C^1$-conforming elements \cite{papanicolopulos2009threedimensional, zervos2009two, manzari2013c1} meet it directly, but are complicated to construct and require high-order polynomials, hence a large number of degrees of freedom. A more practical approach is the $H^2$-nonconforming finite element method \cite{li2017two, li2021h2korns, chen2023robust, liao2023taylorhood, chen2025nonconforming}: since the well-posedness of the SGE model rests on the $H^2$-Korn inequality, it suffices to enforce just enough inter-element continuity for its discrete counterpart to hold, which permits lower polynomial degrees and fewer degrees of freedom. The $C^0$ interior penalty method \cite{brenner2011c0, ventura2021c0} uses standard $C^0$ elements together with penalty terms that weakly enforce $C^1$ continuity; it is easy to implement and amenable to local refinement, but is suboptimal for lower-order elements and requires careful tuning of the penalty parameter. Isogeometric analysis \cite{fischer2011isogeometric, rudraraju2014threedimensional, niiranen2016variational, hosseini20223d} uses B-spline or NURBS bases, which provide arbitrary higher-order continuity; it is effective on simple, regular geometries, but remains challenging on general domains.

All of these methods are displacement based. A conceptually different approach is the mixed finite element (MFE) formulation \cite{shu1999finite, zybell2012threedimensional}, typically of Hu--Washizu type: the displacement gradient is introduced as an independent variable and coupled back to the displacement by a Lagrange multiplier, which yields a second-order system requiring only $C^0$ elements. The price is a larger linear system and, as shown in \cite{riesselmann2020threefield}, the nontrivial task of finding a stable element pair. Of the two remedies proposed, the splitting method \cite{riesselmann2021rotfree} admits any inf-sup stable Stokes pair but is restricted to simply connected domains with homogeneous Dirichlet data and still needs a tuned stabilization term as $\iota \to 0$, while the engineering-oriented refinement \cite{riesselmann2024efficient} reduces the number of unknowns and supports general boundary conditions, but is not accompanied by a convergence analysis.

The second challenge is robustness with respect to $\iota$ and $\lambda$. Equation \eqref{eq:SGE} is a singularly perturbed problem in $\iota$; at $\iota = 0$ it reduces to classical elasticity,
\begin{align} \label{eq:classical-elasticity}
\begin{aligned}
    - \mu \Delta \vect{u}_0 - (\lambda + \mu) \grad \divg \vect{u}_0 &= \vect{f} & &\text{in } \Omega, \\
    \vect{u}_0 &= \vect{0} & &\text{on } \partial \Omega,
\end{aligned}
\end{align}
whose solution cannot inherit the higher-order Dirichlet boundary condition; for $\iota \ll 1$ the latter is instead absorbed into a boundary layer of thickness $O(\iota)$. As $\lambda \to \infty$, the material approaches the incompressible limit and lower-order conforming methods may suffer volumetric locking. Rigorous parameter-robust analyses have been developed mainly within the $H^2$-nonconforming family: \cite{li2017two, li2021h2korns} establish error estimates uniform in $\iota$, and \cite{chen2023robust, liao2023taylorhood}, through an auxiliary variable $p = \lambda \divg \vect{u}$, in both $\iota$ and $\lambda$. For $\iota \ll 1$, however, the uniform rate of these methods degrades to $O(h^{1/2})$, which is to be expected of displacement-based methods: imposing the higher-order Dirichlet boundary condition strongly on the discrete space creates a numerical boundary layer analogous to the physical one, but of thickness governed by $h$ rather than $\iota$. Nitsche's technique \cite{chen2025nonconforming} recovers the optimal rate by imposing that condition weakly through penalty terms, at the price of additional parameter to be tuned carefully.

\subsection{Framework of This Work}

Working within the MFE framework, we combine the staggered discontinuous Galerkin (SDG) method with hybridization to obtain a four-field scheme for strain gradient elasticity, together with a rigorous theoretical analysis. The proposed method
\begin{itemize}
    \item offers a direct approximation of the total stress;
    \item converges at the optimal rate, uniformly in $\iota$ and $\lambda$;
    \item requires no stabilization or penalty parameter;
    \item leaves, after hybridization and local static condensation, only two multipliers on the mesh skeleton as globally coupled unknowns.
\end{itemize}

Our main reason for adopting the MFE framework, beyond the direct approximation of quantities of practical interest such as the total stress, is that it incorporates the higher-order Dirichlet boundary condition into the variational formulation rather than imposing it strongly on the discrete space; the numerical boundary layer is thereby avoided and the optimal rate is retained even for $\iota \ll 1$. In place of the Hu--Washizu-type formulation of existing MFE methods, we adopt a first-order reformulation of \eqref{eq:SGE} with the total stress together with the scaled Cauchy and hyper stresses as auxiliary unknowns.

The SDG method was introduced by Chung and Engquist \cite{chung2006optimal, chung2009optimal} for wave equations and has since been developed for a wide range of problems; we mention only \cite{kim2013staggered, lee2016analysis, chung2017analysis, zhao2018staggered, kim2020staggered, zhao2023lockingfree}. Its core idea is to perform an Alfeld split of the primal triangulation to obtain a dual mesh, and then to use equal-order polynomial spaces with staggered continuity across the primal and the dual edges. These spaces suit our purposes in four respects: they carry enough redundancy so that inf-sup stability survives the strong imposition of symmetry; their equal order facilitates balancing the accuracy between the different variables; their staggered continuity makes integration by parts produce natural rather than artificial numerical fluxes, so that no stabilization is needed; and they are broken $H^1$- or $H(\divg)$-conforming, hence naturally compatible with local static condensation.

The remainder of the paper is organized as follows. \Cref{sec:2} derives the first-order reformulation. \Cref{sec:3} introduces the discrete spaces --- SDG spaces for the displacement and the total stress, Raviart--Thomas pairs for the two scaled stresses, with symmetry imposed strongly on all three stresses --- and the resulting scheme. \Cref{sec:4} establishes an inf-sup condition on the symmetric subspace and a discrete Korn inequality, and derives the stability estimates. \Cref{sec:5} gives first an optimal-order error estimate that tracks the regularity of $\vect{u}$ and the parameters explicitly, and then, using the regularity theory of the SGE model together with projection estimates under low regularity, derives an estimate uniform in $\iota$ and $\lambda$. \Cref{sec:6} develops a hybridized scheme in which the continuity of the total stress and the scaled hyper stress is relaxed and reimposed through Lagrange multipliers; inf-sup conditions are developed for the two multipliers, from which their stability follows, and local static condensation is shown to reduce the global system to these multipliers alone on the mesh skeleton. A modification suitable for the incompressible limit is given in \Cref{sm:incompressible}. \Cref{sec:7} reports two numerical experiments, one with a smooth solution and one with a strong boundary layer, confirming the predicted convergence rates and the robustness in $\lambda$ and $\iota$.

To the best of our knowledge, this is the first four-field mixed scheme for strain gradient elasticity to be supported by a rigorous analysis. Compared with displacement-based methods, it avoids the numerical boundary layer and retains the optimal convergence rate even for small $\iota$; compared with existing MFE methods, it rests on a different reformulation which, together with a carefully matched set of discrete spaces, is what makes parameter-uniform stability and convergence possible. Moreover, it offers a direct approximation of the total stress that is an important quantity of practical interest. For reasons of space we treat only the two-dimensional case with homogeneous doubly clamped boundary conditions; the extension to three dimensions and to general boundary conditions is left to future work.

\section{A First-Order Reformulation of Strain Gradient Elasticity} \label{sec:2}

\subsection{Notations}

We use index component notation together with the Einstein summation convention throughout; all indices range over $\{1, 2\}$, and $\matr{I}$ denotes the identity matrix. For a vector field $\vect{v}$ we set
\begin{align*}
    (\grad \vect{v})_{ij} \defeq \partial_j v_i, \qquad
    \divg \vect{v} \defeq \partial_i v_i, \qquad
    (\Delta \vect{v})_i \defeq \partial_j \partial_j v_i, \qquad
    \curl \vect{v} \defeq
    \begin{pmatrix}
        -\partial_2 v_1 & \partial_1 v_1 \\
        -\partial_2 v_2 & \partial_1 v_2
    \end{pmatrix},
\end{align*}
and, for a matrix field $\matr{\tau}$ and a vector $\vect{n}$,
\begin{align*}
    (\grad \matr{\tau})_{ijk} &\defeq \partial_k \tau_{ij}, &
    (\divg \matr{\tau})_i &\defeq \partial_j \tau_{ij}, &
    (\Delta \matr{\tau})_{ij} &\defeq \partial_k \partial_k \tau_{ij}, \\
    \tr \matr{\tau} &\defeq \tau_{ii}, &
    \dev \matr{\tau} &\defeq \matr{\tau} - \tfrac{1}{2} (\tr \matr{\tau}) \matr{I}, &
    (\matr{\tau} \cdot \vect{n})_i &\defeq \tau_{ij} n_j.
\end{align*}
For a third-order tensor field $\tens{\chi}$, we set
\begin{align*}
    (\divg \tens{\chi})_{ij} \defeq \partial_k \chi_{ijk}, \qquad
    (\tens{\chi} \cdot \vect{n})_{ij} \defeq \chi_{ijk} n_k, \qquad
    \tens{\chi}_{ij\bullet} \defeq (\chi_{ij1}, \chi_{ij2})^{\trans},
\end{align*}
the last being the vector obtained by fixing the first two indices of $\tens{\chi}$.

\subsection{Strain Gradient Elasticity}

We briefly recall the derivation of the Aifantis's SGE model. With the strain tensor $\matr{\varepsilon} = \symgrad \vect{u} \defeq \frac{1}{2} ( \grad \vect{u} + (\grad \vect{u})^{\trans} )$ and the fourth-order elasticity tensor
\begin{align} \label{eq:elasticity-tensor}
    \tens{C} \dvdots \matr{\varepsilon} \defeq 2\mu \matr{\varepsilon} + \lambda (\tr \matr{\varepsilon}) \matr{I},
\end{align}
the total stress $\matr{\sigma}$ is related to the strain by the modified Hooke's law,
\begin{align} \label{eq:constitutive-relation}
    \matr{\sigma} \defeq (1 - \iota^2 \Delta) \, \tens{C} \dvdots \matr{\varepsilon},
\end{align}
and, given a body force $\vect{f}$, the equilibrium equation reads
\begin{align} \label{eq:equilibrium}
    - \divg \matr{\sigma} = \vect{f} \qquad \text{in } \Omega.
\end{align}
Substituting \eqref{eq:elasticity-tensor} and \eqref{eq:constitutive-relation} into \eqref{eq:equilibrium} recovers the displacement form \eqref{eq:SGE} of the SGE equation.

\subsection{First-Order Reformulation}

We now reformulate \eqref{eq:SGE} as a first-order system in which the total stress is retained as unknowns. Define the square root $\tens{C}^{1/2}$ of the elasticity tensor and its inverse $\mathcal{A}^{1/2}$ by
\begin{align} \label{eq:A^1/2}
\begin{aligned}
    \tens{C}^{\frac{1}{2}} \dvdots \matr{\varepsilon} &\defeq \sqrt{2\mu} \, \dev \matr{\varepsilon} + \frac{1}{2} \sqrt{2\lambda + 2\mu} (\tr \matr{\varepsilon}) \matr{I}, \\
    \mathcal{A}^{\frac{1}{2}} \matr{\sigma} &\defeq \frac{1}{\sqrt{2\mu}} \dev \matr{\sigma} + \frac{1}{2 \sqrt{2\lambda + 2\mu}} (\tr \matr{\sigma}) \matr{I};
\end{aligned}
\end{align}
a direct computation shows that $\tens{C}^{1/2} \dvdots ( \tens{C}^{1/2} \dvdots \matr{\varepsilon} ) = \tens{C} \dvdots \matr{\varepsilon}$ and that $\mathcal{A}^{1/2}$ inverts $\tens{C}^{1/2}$. Introducing the two auxiliary variables
\begin{align} \label{eq:auxiliary-variables}
    \widetilde{\matr{\sigma}} \defeq \tens{C}^{\frac{1}{2}} \dvdots \matr{\varepsilon}, \qquad
    \tens{\chi} \defeq \iota \grad \widetilde{\matr{\sigma}},
\end{align}
which together with $\vect{u}$ and $\matr{\sigma}$ constitute the four fields of the proposed formulation, the constitutive relation \eqref{eq:constitutive-relation} becomes
\begin{align} \label{eq:constitutive-relation-reform}
    \matr{\sigma} = \tens{C}^{\frac{1}{2}} \dvdots \left( \widetilde{\matr{\sigma}} - \iota \divg \tens{\chi} \right).
\end{align}
Applying $\mathcal{A}^{1/2}$ to \eqref{eq:auxiliary-variables} and \eqref{eq:constitutive-relation-reform} and combining with \eqref{eq:equilibrium}, we arrive at the first-order system
\begin{align}
    \mathcal{A}^{\frac{1}{2}} \widetilde{\matr{\sigma}} - \symgrad \vect{u} &= \matr{0}, \label{eq:first-order-SGE-1} \\
    \tens{\chi} - \iota \grad \widetilde{\matr{\sigma}} &= \tens{0}, \label{eq:first-order-SGE-2} \\
    \widetilde{\matr{\sigma}} - \mathcal{A}^{\frac{1}{2}} \matr{\sigma} - \iota \divg \tens{\chi} &= \matr{0}, \label{eq:first-order-SGE-3} \\
    - \divg \matr{\sigma} &= \vect{f}. \label{eq:first-order-SGE-4}
\end{align}

\begin{remark} \label{rmk:symmetry-and-homogeneous-BC}
    Symmetry propagates through the system: the symmetry of $\symgrad \vect{u}$ implies $\widetilde{\matr{\sigma}}^{\trans} = \widetilde{\matr{\sigma}}$ via \eqref{eq:first-order-SGE-1}, which in turn yields $\tens{\chi}_{ij\bullet} = \tens{\chi}_{ji\bullet}$ and $\matr{\sigma}^{\trans} = \matr{\sigma}$ through \eqref{eq:first-order-SGE-2} and \eqref{eq:first-order-SGE-3}, respectively. Moreover, the homogeneous boundary conditions \eqref{eq:homogeneous-BC} imply $\matr{\varepsilon} = \matr{0}$ and hence $\widetilde{\matr{\sigma}} = \matr{0}$ on $\partial \Omega$.
\end{remark}

\begin{remark}
    The auxiliary variables admit a direct mechanical interpretation. In the general strain gradient elasticity theory, the material response is described by two stress measures, the Cauchy stress $\matr{\sigma}^c \defeq \tens{C} \dvdots \matr{\varepsilon}$ and the hyper stress $\tens{\tau} \defeq \iota^2 \grad \matr{\sigma}^c$. By \eqref{eq:auxiliary-variables}, $\widetilde{\matr{\sigma}} = \mathcal{A}^{1/2} \matr{\sigma}^c$ and $\tens{\chi} = \iota^{-1} \mathcal{A}^{1/2} \tens{\tau}$, so that $\widetilde{\matr{\sigma}}$ and $\tens{\chi}$ are scaled versions of the Cauchy stress and the hyper stress, where the factor $\iota^{-1}$ balances the magnitude of the hyper stress. As $\lambda \to \infty$, the energy norm induced by $\mathcal{A}$ degenerates to its deviatoric part alone; the scaling by $\mathcal{A}^{1/2}$ confines this degeneration to the total stress $\matr{\sigma}$, where a discrete Korn inequality recovers full $L^2$ control, as detailed in \cref{sec:4}.
\end{remark}

\section{Discrete Scheme} \label{sec:3}

\subsection{Mesh Partition and Notations}

Let $\mathcal{T}^{pr}_h$ be a shape-regular triangulation of $\Omega$, which we call the primal mesh; for $K \in \mathcal{T}^{pr}_h$ let $h_K$ denote the diameter of $K$, and let $h \defeq \max_{K \in \mathcal{T}^{pr}_h} h_K$ be the mesh size. The edges of $\mathcal{T}^{pr}_h$ are called primal edges; we denote by $\mathcal{F}^{pr}_h$ the set of all primal edges, by $\mathcal{F}^{pr, o}_h$ its subset of interior ones, and by $\mathcal{F}^{\partial}_h$ the set of edges lying on $\partial \Omega$, so that $\mathcal{F}^{pr}_h = \mathcal{F}^{pr, o}_h \cup \mathcal{F}^{\partial}_h$.

We perform an Alfeld split on $\mathcal{T}^{pr}_h$ to obtain a refined triangulation $\mathcal{T}_h$: each $K \in \mathcal{T}^{pr}_h$ is divided into three sub-triangles by connecting its barycenter to its three vertices. The newly generated edges are called dual edges, and their set is denoted by $\mathcal{F}^{dl}_h$; we write $\mathcal{F}_h = \mathcal{F}^{pr}_h \cup \mathcal{F}^{dl}_h$ for the set of all edges of $\mathcal{T}_h$. For $e \in \mathcal{F}^{pr}_h$, let $D(e)$ denote the union of the sub-triangles in $\mathcal{T}_h$ having $e$ as an edge; since every sub-triangle has exactly one primal edge, the collection $\mathcal{T}^{dl}_h \defeq \{ D(e) : e \in \mathcal{F}^{pr}_h \}$ is again a partition of $\Omega$, which we call the dual mesh. The construction is illustrated in \cref{fig:mesh}.

\begin{figure}[htbp]
    \centering
    \begin{tikzpicture}[scale=1.15]
        \coordinate (A) at (0,0);
        \coordinate (B) at (3,0);
        \coordinate (C) at (1.1,2.1);
        \coordinate (D) at (1.9,-1.9);
        \coordinate (b1) at (1.3667,0.7);
        \coordinate (b2) at (1.6333,-0.6333);
        \fill[gray!22] (A) -- (B) -- (b1) -- cycle;
        \fill[gray!22] (A) -- (B) -- (b2) -- cycle;
        \draw[dashed] (b1) -- (A); \draw[dashed] (b1) -- (B); \draw[dashed] (b1) -- (C);
        \draw[dashed] (b2) -- (A); \draw[dashed] (b2) -- (B); \draw[dashed] (b2) -- (D);
        \draw[thick] (A) -- (B) -- (C) -- cycle;
        \draw[thick] (A) -- (B) -- (D) -- cycle;
        \fill (b1) circle (1.3pt); \fill (b2) circle (1.3pt);
        \node[fill=white,inner sep=1pt] at (1,0.02) {\small $e$};
        \node[fill=gray!22,inner sep=1pt] at (1.5,0.30) {\small $D(e)$};
        \node at (0.90,1.00) {\small $K$};
        \node at (1.5,-2.45) {(a) primal and dual meshes};
    \end{tikzpicture}
    \hspace{1.1cm}
    \begin{tikzpicture}[scale=1.15]
        \coordinate (A) at (0,0);
        \coordinate (B) at (3,0);
        \coordinate (C) at (1.1,2.1);
        \coordinate (D) at (1.9,-1.9);
        \coordinate (b1) at (1.3667,0.7);
        \coordinate (b2) at (1.6333,-0.6333);
        \draw[dashed] (b1) -- (A); \draw[dashed] (b1) -- (B); \draw[dashed] (b1) -- (C);
        \draw[dashed] (b2) -- (A); \draw[dashed] (b2) -- (B); \draw[dashed] (b2) -- (D);
        \draw[thick] (A) -- (B) -- (C) -- cycle;
        \draw[thick] (A) -- (B) -- (D) -- cycle;
        \fill (1.5,0) circle (2.4pt);
        \fill (0.55,1.05) circle (2.4pt);
        \fill (2.05,1.05) circle (2.4pt);
        \fill (0.95,-0.95) circle (2.4pt);
        \fill (2.45,-0.95) circle (2.4pt);
        \draw[fill=white] (0.6833,0.35) +(-2.4pt,-2.4pt) rectangle +(2.4pt,2.4pt);
        \draw[fill=white] (2.1833,0.35) +(-2.4pt,-2.4pt) rectangle +(2.4pt,2.4pt);
        \draw[fill=white] (1.2333,1.4) +(-2.4pt,-2.4pt) rectangle +(2.4pt,2.4pt);
        \draw[fill=white] (0.8167,-0.3167) +(-2.4pt,-2.4pt) rectangle +(2.4pt,2.4pt);
        \draw[fill=white] (2.3167,-0.3167) +(-2.4pt,-2.4pt) rectangle +(2.4pt,2.4pt);
        \draw[fill=white] (1.7667,-1.2667) +(-2.4pt,-2.4pt) rectangle +(2.4pt,2.4pt);
        \node at (1.5,-2.45) {(b) staggered continuity};
    \end{tikzpicture}
    \caption{Left: two primal triangles sharing the primal edge $e$. Primal edges are drawn as solid lines and dual edges, generated by the Alfeld split, as dashed lines. The shaded region is the dual element $D(e)$. Right: the staggered continuity constraints, those of $U_h$ on the primal edges ($\bullet$) and those of $\Sigma_h$ on the dual edges ($\square$).}
    \label{fig:mesh}
\end{figure}
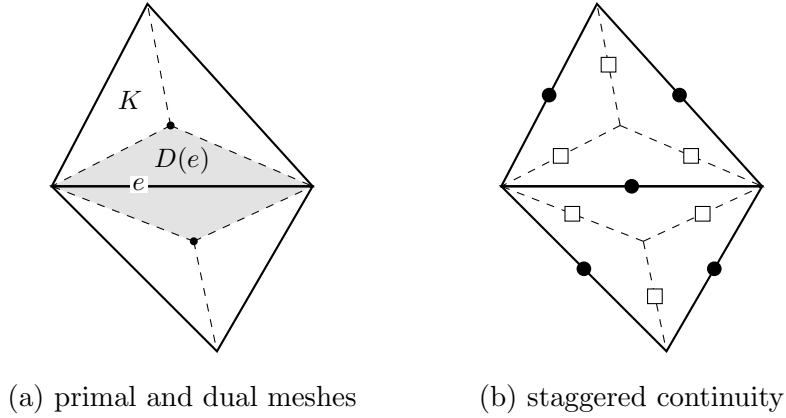

For each $e \in \mathcal{F}_h$, let $h_e$ denote the length of $e$ and $\vect{n}_e$ the unit normal vector to $e$, oriented arbitrarily but fixed for an interior edge and outward for a boundary edge; we write $\vect{n}$ when no confusion arises. For an interior edge $e$ shared by $K^+, K^- \in \mathcal{T}_h$, with $\vect{n}$ oriented from $K^+$ to $K^-$, the jump of a scalar function $\phi$ across $e$ is $\jump{\phi}_{|_e} \defeq \phi|_{K^+} - \phi|_{K^-}$; for a boundary edge we set $\jump{\phi}_{|_e} \defeq \phi$. The same notation is used for vector- and tensor-valued functions.

Let $\omega$ denote either an element $K \in \mathcal{T}_h$, an edge $e \in \mathcal{F}_h$, or the whole domain $\Omega$. We write $L^2 (\omega)$, $H (\divg; \omega)$ and $H^k (\omega)$ for the space of square-integrable functions, the space of square-integrable vector fields with square-integrable divergence, and the Sobolev space of functions with square-integrable derivatives up to order $k$, respectively, and $\mathcal{P}_k (\omega)$ and $\widetilde{\mathcal{P}}_k (\omega)$ for the spaces of polynomials of degree at most $k$ and of homogeneous polynomials of degree $k$ on $\omega$. These definitions are extended to vector- and tensor-valued functions by specifying the target space; for example, $L^2 (\omega; \mathbb{R}^2)$ denotes the space of vector fields with each component in $L^2 (\omega)$. Let $(\cdot, \cdot)_{\omega}$ denote the $L^2$ inner product on $\omega$, and $\lVert \cdot \rVert_{\omega}$ and $\lVert \cdot \rVert_{k, \omega}$ the $L^2$ and $H^k$ norms; the subscript is omitted when $\omega = \Omega$. The broken inner products $(\cdot, \cdot)_{\mathcal{T}_h} \defeq \sum_{K \in \mathcal{T}_h} (\cdot, \cdot)_K$ and $(\cdot, \cdot)_{\mathcal{F}_h} \defeq \sum_{e \in \mathcal{F}_h} (\cdot, \cdot)_e$ are defined analogously over any subset of $\mathcal{T}_h$ or $\mathcal{F}_h$. Finally, $\lesssim$ denotes an inequality up to a generic positive constant that is independent of the mesh size $h$, the length scale parameter $\iota$ and the Lam\'e constant $\lambda$, but may depend on $\mu$ and on the shape-regularity of the mesh.

\subsection{Discrete Spaces}

The four fields are discretized by two different pairs of spaces: the displacement and the total stress in staggered discontinuous Galerkin (SDG) spaces, and the scaled Cauchy stress and the scaled hyper stress in a Raviart--Thomas (RT) pair.

\subsubsection{SDG Spaces}

We first recall the SDG spaces introduced in \cite{chung2009optimal, kim2013staggered}. Let $U_h$ be the vector-valued piecewise polynomial space with continuity enforced across the primal edges, and $\Sigma_h$ the matrix-valued one with normal continuity enforced across the dual edges,
\begin{align*}
    U_h &\defeq \{ \vect{v} \in L^2 (\Omega; \mathbb{R}^2) : \vect{v}|_K \in \mathcal{P}_{k} (K; \mathbb{R}^2),\ \forall K \in \mathcal{T}_h;\ \jump{\vect{v}}_{|_e} = \vect{0},\ \forall e \in \mathcal{F}^{pr, o}_h \}, \\
    \Sigma_h &\defeq \{ \matr{\tau} \in L^2 (\Omega; \mathbb{R}^{2 \times 2}) : \matr{\tau}|_K \in \mathcal{P}_{k} (K; \mathbb{R}^{2 \times 2}),\ \forall K \in \mathcal{T}_h;\ \jump{\matr{\tau} \cdot \vect{n}}_{|_e} = \vect{0},\ \forall e \in \mathcal{F}^{dl}_h \},
\end{align*}
so that $U_h$ is $H^1$-conforming on each dual element and $\Sigma_h$ is $H(\divg)$-conforming on each primal element. The degrees of freedom of $U_h$ are $(\vect{v}, \vect{\phi})_e$ with $\vect{\phi} \in \mathcal{P}_k (e; \mathbb{R}^2)$, $e \in \mathcal{F}^{pr}_h$, and $(\vect{v}, \vect{\phi})_K$ with $\vect{\phi} \in \mathcal{P}_{k - 1} (K; \mathbb{R}^2)$, $K \in \mathcal{T}_h$; those of $\Sigma_h$ are $(\matr{\tau} \cdot \vect{n}, \vect{\phi})_e$ with $\vect{\phi} \in \mathcal{P}_k (e; \mathbb{R}^2)$, $e \in \mathcal{F}^{dl}_h$, and $(\matr{\tau}, \matr{\phi})_K$ with $\matr{\phi} \in \mathcal{P}_{k - 1} (K; \mathbb{R}^{2 \times 2})$, $K \in \mathcal{T}_h$. Their staggered pattern is shown in \cref{fig:mesh}(b).

The discrete spaces for the displacement and the total stress are obtained from $U_h$ and $\Sigma_h$ by imposing the boundary condition and the symmetry constraint, respectively:
\begin{align*}
    U_{h, 0} &\defeq \{ \vect{v}_h \in U_h : \vect{v}_h |_e = \vect{0},\ \forall e \in \mathcal{F}^{\partial}_h \}, \\
    \Sigma^s_h &\defeq \{ \matr{\tau}_h \in \Sigma_h : \matr{\tau}^{\trans}_h = \matr{\tau}_h \}.
\end{align*}

\subsubsection{RT Elements}

Let $RT_k (\mathcal{T}_h)$ be the Raviart--Thomas space \cite{boffi2013mixed} of order $k$ on $\mathcal{T}_h$ and $DG_k (\mathcal{T}_h)$ the discontinuous polynomial space of order $k$,
\begin{align*}
    RT_k (\mathcal{T}_h) &\defeq \{ \vect{w} \in H (\divg; \Omega) : \vect{w}|_K \in \mathcal{P}_k (K; \mathbb{R}^2) \oplus \vect{x} \widetilde{\mathcal{P}}_k (K),\ \forall K \in \mathcal{T}_h \}, \\
    DG_k (\mathcal{T}_h) &\defeq \{ q \in L^2 (\Omega) : q|_K \in \mathcal{P}_k (K),\ \forall K \in \mathcal{T}_h \},
\end{align*}
the degrees of freedom of $RT_k (\mathcal{T}_h)$ being $(\vect{w} \cdot \vect{n}, \phi)_e$ with $\phi \in \mathcal{P}_k (e)$, $e \in \mathcal{F}_h$, and $(\vect{w}, \vect{\phi})_K$ with $\vect{\phi} \in \mathcal{P}_{k - 1} (K; \mathbb{R}^2)$, $K \in \mathcal{T}_h$. The discrete space for the scaled hyper stress is obtained by replicating $RT_k (\mathcal{T}_h)$ along the first two indices and imposing symmetry, and that for the scaled Cauchy stress by replicating $DG_k (\mathcal{T}_h)$ in the same way:
\begin{align*}
    X^s_h &\defeq \{ \tens{\psi} \in L^2 (\Omega; \mathbb{R}^{2 \times 2 \times 2}) : \tens{\psi}_{ij\bullet} \in RT_k (\mathcal{T}_h);\ \tens{\psi}_{12\bullet} = \tens{\psi}_{21\bullet} \}, \\
    \widetilde{\Sigma}^s_h &\defeq \{ \widetilde{\matr{\tau}} \in L^2 (\Omega; \mathbb{R}^{2 \times 2}) : \widetilde{\tau}_{ij} \in DG_k (\mathcal{T}_h);\ \widetilde{\tau}_{12} = \widetilde{\tau}_{21} \}.
\end{align*}

\subsection{The Scheme}

Guided by the staggered continuity of $U_h$ and $\Sigma_h$, we introduce the discrete bilinear forms
\begin{align}
    b_h (\matr{\tau}, \vect{v}) &\defeq (\matr{\tau}, \grad \vect{v})_{\mathcal{T}_h} - (\matr{\tau} \cdot \vect{n}, \jump{\vect{v}})_{\mathcal{F}^{dl}_h}, \label{eq:b_h} \\
    b^*_h (\vect{v}, \matr{\tau}) &\defeq - (\vect{v}, \divg \matr{\tau})_{\mathcal{T}_h} + (\vect{v}, \jump{\matr{\tau} \cdot \vect{n}})_{\mathcal{F}^{pr}_h}. \label{eq:b*_h}
\end{align}
A direct calculation gives the discrete integration-by-parts identity
\begin{align}
    b_h (\matr{\tau}_h, \vect{v}_h) = b^*_h (\vect{v}_h, \matr{\tau}_h), \qquad \forall \matr{\tau}_h \in \Sigma_h, \ \vect{v}_h \in U_h. \label{eq:adjointness}
\end{align}
Then, the discrete scheme reads as follows.

\paragraph{Discrete scheme}
Find $(\vect{u}_h, \matr{\sigma}_h, \widetilde{\matr{\sigma}}_h, \tens{\chi}_h) \in U_{h, 0} \times \Sigma^s_h \times \widetilde{\Sigma}^s_h \times X^s_h$ such that
\begin{align}
    (\mathcal{A}^{\frac{1}{2}} \widetilde{\matr{\sigma}}_h, \matr{\tau}_h) - b^*_h (\vect{u}_h, \matr{\tau}_h) &= 0, & &\forall \matr{\tau}_h \in \Sigma^s_h, \label{eq:discrete-scheme-1} \\
    (\tens{\chi}_h, \tens{\psi}_h) + \iota (\widetilde{\matr{\sigma}}_h, \divg \tens{\psi}_h) &= 0, & &\forall \tens{\psi}_h \in X^s_h, \label{eq:discrete-scheme-2} \\
    (\widetilde{\matr{\sigma}}_h, \widetilde{\matr{\tau}}_h) - (\mathcal{A}^{\frac{1}{2}} \matr{\sigma}_h, \widetilde{\matr{\tau}}_h) - \iota (\divg \tens{\chi}_h, \widetilde{\matr{\tau}}_h) &= 0, & &\forall \widetilde{\matr{\tau}}_h \in \widetilde{\Sigma}^s_h, \label{eq:discrete-scheme-3} \\
    b_h (\matr{\sigma}_h, \vect{v}_h) &= (\vect{f}, \vect{v}_h), & &\forall \vect{v}_h \in U_{h, 0}. \label{eq:discrete-scheme-4}
\end{align}

The scheme is consistent: testing the first-order system \eqref{eq:first-order-SGE-1}--\eqref{eq:first-order-SGE-4} with $ (\matr{\tau}_h, \tens{\psi}_h, \widetilde{\matr{\tau}}_h, \vect{v}_h) \in \Sigma^s_h \times X^s_h \times \widetilde{\Sigma}^s_h \times U_{h, 0}$ and integrating by parts gives the variational formulation
\begin{align}
    (\mathcal{A}^{\frac{1}{2}} \widetilde{\matr{\sigma}}, \matr{\tau}_h) - b^*_h (\vect{u}, \matr{\tau}_h) &= 0, & &\forall \matr{\tau}_h \in \Sigma^s_h, \label{eq:variation-formulation-1} \\
    (\tens{\chi}, \tens{\psi}_h) + \iota (\widetilde{\matr{\sigma}}, \divg \tens{\psi}_h) &= 0, & &\forall \tens{\psi}_h \in X^s_h, \label{eq:variation-formulation-2} \\
    (\widetilde{\matr{\sigma}}, \widetilde{\matr{\tau}}_h) - (\mathcal{A}^{\frac{1}{2}} \matr{\sigma}, \widetilde{\matr{\tau}}_h) - \iota (\divg \tens{\chi}, \widetilde{\matr{\tau}}_h) &= 0, & &\forall \widetilde{\matr{\tau}}_h \in \widetilde{\Sigma}^s_h, \label{eq:variation-formulation-3} \\
    b_h (\matr{\sigma}, \vect{v}_h) &= (\vect{f}, \vect{v}_h), & &\forall \vect{v}_h \in U_{h, 0}, \label{eq:variation-formulation-4}
\end{align}
where the boundary term $\iota (\widetilde{\matr{\sigma}}, \tens{\psi}_h \cdot \vect{n})_{\mathcal{F}^{\partial}_h}$ arising in \eqref{eq:variation-formulation-2} vanishes since $\widetilde{\matr{\sigma}} |_{\partial \Omega} = \matr{0}$ by \cref{rmk:symmetry-and-homogeneous-BC}.

\section{Stability Analysis} \label{sec:4}

\subsection{Preliminaries}

For $\vect{v}_h \in U_h$ and $\matr{\tau}_h \in \Sigma_h$, define the discrete norms
\begin{align*}
    \lVert \vect{v}_h \rVert^2_{L^2, h} &\defeq \sum_{K \in \mathcal{T}_h} \lVert \vect{v}_h \rVert^2_{K} + \sum_{e \in \mathcal{F}^{pr}_h} h_e \lVert \vect{v}_h \rVert^2_{e}, \\
    \lVert \vect{v}_h \rVert^2_{H^1, h} &\defeq \sum_{K \in \mathcal{T}_h} \lVert \grad \vect{v}_h \rVert^2_{K} + \sum_{e \in \mathcal{F}^{dl}_h} h_e^{-1} \lVert \jump{\vect{v}_h} \rVert^2_{e}, \\
    \lVert \matr{\tau}_h \rVert^2_{L^2, h} &\defeq \sum_{K \in \mathcal{T}_h} \lVert \matr{\tau}_h \rVert^2_{K} + \sum_{e \in \mathcal{F}^{dl}_h} h_e \lVert \matr{\tau}_h \cdot \vect{n} \rVert^2_{e}, \\
    \lVert \matr{\tau}_h \rVert^2_{\divg, h} &\defeq \sum_{K \in \mathcal{T}_h} \lVert \divg \matr{\tau}_h \rVert^2_{K} + \sum_{e \in \mathcal{F}^{pr, o}_h} h_e^{-1} \lVert \jump{\matr{\tau}_h \cdot \vect{n}} \rVert^2_{e}.
\end{align*}
It was proved in \cite{chung2009optimal, kim2013staggered} that the discrete $L^2$ norms are equivalent to the standard ones, i.e., $\lVert \vect{v}_h \rVert \lesssim \lVert \vect{v}_h \rVert_{L^2, h} \lesssim \lVert \vect{v}_h \rVert$ for $\vect{v}_h \in U_h$ and $\lVert \matr{\tau}_h \rVert \lesssim \lVert \matr{\tau}_h \rVert_{L^2, h} \lesssim \lVert \matr{\tau}_h \rVert$ for $\matr{\tau}_h \in \Sigma_h$. The bilinear forms \eqref{eq:b_h}--\eqref{eq:b*_h} are bounded,
\begin{align*}
    b_h (\matr{\tau}_h, \vect{v}_h) &\lesssim \lVert \matr{\tau}_h \rVert \lVert \vect{v}_h \rVert_{H^1, h}, &
    b^*_h (\vect{v}_h, \matr{\tau}_h) &\lesssim \lVert \vect{v}_h \rVert \lVert \matr{\tau}_h \rVert_{\divg, h}, & &\forall (\vect{v}_h, \matr{\tau}_h) \in U_{h, 0} \times \Sigma_h.
\end{align*}
Standard DG theory \cite{dipietro2012mathematical} gives the discrete Poincar\'e inequality on $U_{h, 0}$:
\begin{align}
    \lVert \vect{v}_h \rVert &\lesssim \lVert \vect{v}_h \rVert_{H^1, h}, & &\forall \vect{v}_h \in U_{h, 0}. \label{eq:poincare-ineq-U_h}
\end{align}
The inverse inequality \cite{ern2004theory} applies on $X^s_h$:
\begin{align}
    \lVert \divg \tens{\psi}_h \rVert &\lesssim h^{-1} \lVert \tens{\psi}_h \rVert, & &\forall \tens{\psi}_h \in X^s_h. \label{eq:inverse-ineq-X_h}
\end{align}
Finally, the following inf-sup conditions were established in \cite{chung2009optimal, kim2013staggered}, with positive constants $\beta_1$ and $\beta_2$ independent of the mesh size $h$:
\begin{align}
    \sup_{\vect{v}_h \in U_{h, 0} \setminus \{ \vect{0} \}} \frac{b_h (\matr{\tau}_h, \vect{v}_h)}{\lVert \vect{v}_h \rVert} &\geq \beta_1 \lVert \matr{\tau}_h \rVert_{\divg, h}, & &\forall \matr{\tau}_h \in \Sigma_h, \label{eq:inf-sup-b_h} \\
    \sup_{\matr{\tau}_h \in \Sigma_h \setminus \{ \matr{0} \}} \frac{b^*_h (\vect{v}_h, \matr{\tau}_h)}{\lVert \matr{\tau}_h \rVert} &\geq \beta_2 \lVert \vect{v}_h \rVert_{H^1, h}, & &\forall \vect{v}_h \in U_{h, 0}. \label{eq:inf-sup-b*_h}
\end{align}

\subsection{Inf-Sup Condition on the Symmetric Subspace}

The inf-sup condition \eqref{eq:inf-sup-b*_h} holds over $\Sigma_h$ but does not transfer to the symmetric subspace $\Sigma^s_h$; we recover it by symmetrizing $\matr{\tau}_h \in \Sigma_h$ with a $\curl$ correction, for which the following Scott--Vogelius pair is required.

\begin{lemma}[Scott--Vogelius pair] \label{lem:SV}
    Let
    \begin{align*}
        W_h &\defeq \{ \vect{w}_h \in H^1 (\Omega; \mathbb{R}^2) : \vect{w}_h|_K \in \mathcal{P}_{k+1} (K; \mathbb{R}^2),\ \forall K \in \mathcal{T}_h \}, \\
        Q_h &\defeq \{ q_h \in L^2 (\Omega) : q_h|_K \in \mathcal{P}_k (K),\ \forall K \in \mathcal{T}_h \}.
    \end{align*}
    If $k \geq 1$, then $\divg : W_h \to Q_h$ is surjective and admits a bounded right inverse: for every $q_h \in Q_h$ there exists $\vect{w}_h \in W_h$ with $\divg \vect{w}_h = q_h$ and $\lVert \vect{w}_h \rVert_{H^1} \lesssim \lVert q_h \rVert$.
\end{lemma}

\begin{proof}
    Set $W_{h, 0} \defeq \{ \vect{w}_h \in W_h : \vect{w}_h|_{\partial \Omega} = \vect{0} \}$ and $Q_{h, 0} \defeq \{ q_h \in Q_h : \int_\Omega q_h = 0 \}$. Since $\mathcal{T}_h$ is the Alfeld split of a shape-regular triangulation, the theory of the Scott--Vogelius pair \cite{guzman2018scottvogelius} shows that for every $q^0_h \in Q_{h, 0}$ there exists $\vect{w}^0_h \in W_{h, 0}$ with $\divg \vect{w}^0_h = q^0_h$ and $\lVert \vect{w}^0_h \rVert_{H^1} \lesssim \lVert q^0_h \rVert$.

    To remove the zero-mean restriction, given $q_h \in Q_h$ set $c \defeq |\Omega|^{-1} \int_\Omega q_h$ and $\vect{w}^1_h \defeq \frac{c}{2} \vect{x} \in W_h$, so that $\divg \vect{w}^1_h = c$ and $\lVert \vect{w}^1_h \rVert_{H^1} \lesssim \lVert q_h \rVert$. Since $q_h - c \in Q_{h, 0}$, the above provides $\vect{w}^0_h \in W_{h, 0}$ with $\divg \vect{w}^0_h = q_h - c$ and $\lVert \vect{w}^0_h \rVert_{H^1} \lesssim \lVert q_h - c \rVert \lesssim \lVert q_h \rVert$. The field $\vect{w}_h \defeq \vect{w}^0_h + \vect{w}^1_h \in W_h$ then satisfies $\divg \vect{w}_h = q_h$ and $\lVert \vect{w}_h \rVert_{H^1} \lesssim \lVert q_h \rVert$.
\end{proof}

\begin{lemma}[Symmetrization operator] \label{lem:symmetrization}
    There exists an operator $\mathcal{S} : \Sigma_h \to \Sigma^s_h$ such that
    \begin{align}
        \lVert \mathcal{S} \matr{\tau}_h \rVert &\lesssim \lVert \matr{\tau}_h \rVert, & &\forall \matr{\tau}_h \in \Sigma_h, \label{eq:symmetrization-1} \\
        b^*_h (\vect{v}_h, \mathcal{S} \matr{\tau}_h) &= b^*_h (\vect{v}_h, \matr{\tau}_h), & &\forall (\vect{v}_h, \matr{\tau}_h) \in U_{h, 0} \times \Sigma_h. \label{eq:symmetrization-2}
    \end{align}
\end{lemma}

\begin{proof}
    Fix $\matr{\tau}_h \in \Sigma_h$, set $q_h \defeq (\matr{\tau}_h)_{21} - (\matr{\tau}_h)_{12}$, and let $\vect{w}_h \in W_h$ be as in \cref{lem:SV}, so that $\divg \vect{w}_h = q_h$ and $\lVert \vect{w}_h \rVert_{H^1} \lesssim \lVert q_h \rVert$. Define $\mathcal{S} \matr{\tau}_h \defeq \matr{\tau}_h + \curl \vect{w}_h$.

    Each entry of $\curl \vect{w}_h$ is piecewise $\mathcal{P}_k$, and $\vect{w}_h \in H^1 (\Omega; \mathbb{R}^2)$ implies $\curl \vect{w}_h \in H(\divg; \Omega)$; hence $\mathcal{S} \matr{\tau}_h \in \Sigma_h$. Since $(\curl \vect{w}_h)_{12} - (\curl \vect{w}_h)_{21} = \divg \vect{w}_h$,
    \begin{align*}
        (\mathcal{S} \matr{\tau}_h)_{12} - (\mathcal{S} \matr{\tau}_h)_{21} = \left[ (\matr{\tau}_h)_{12} - (\matr{\tau}_h)_{21} \right] + \divg \vect{w}_h = -q_h + q_h = 0,
    \end{align*}
    and therefore $\mathcal{S} \matr{\tau}_h \in \Sigma^s_h$. Property \eqref{eq:symmetrization-1} follows from
    \begin{align*}
        \lVert \mathcal{S} \matr{\tau}_h \rVert \leq \lVert \matr{\tau}_h \rVert + \lVert \curl \vect{w}_h \rVert \lesssim \lVert \matr{\tau}_h \rVert + \lVert \vect{w}_h \rVert_{H^1} \lesssim \lVert \matr{\tau}_h \rVert + \lVert q_h \rVert \lesssim \lVert \matr{\tau}_h \rVert.
    \end{align*}
    For \eqref{eq:symmetrization-2}, the conformity $\curl \vect{w}_h \in H(\divg; \Omega)$ gives $\jump{\curl \vect{w}_h \cdot \vect{n}}_{|_e} = \vect{0}$ across every interior edge, while $\divg (\curl \vect{w}_h) = \vect{0}$ pointwise. Consequently, for every $\vect{v}_h \in U_{h, 0}$,
    \begin{align*}
        b^*_h (\vect{v}_h, \curl \vect{w}_h) = - (\vect{v}_h, \divg (\curl \vect{w}_h))_{\mathcal{T}_h} + (\vect{v}_h, \jump{\curl \vect{w}_h \cdot \vect{n}})_{\mathcal{F}^{pr}_h} = 0,
    \end{align*}
    which follows \eqref{eq:symmetrization-2}.
\end{proof}

\begin{theorem}[Inf-sup condition on the symmetric subspace] \label{thm:inf-sup-b*_h-sym}
    There exists a constant $\beta > 0$ independent of the mesh size $h$ such that
    \begin{align}
        \sup_{\matr{\tau}_h \in \Sigma^s_h \setminus \{ \matr{0} \}} \frac{b^*_h (\vect{v}_h, \matr{\tau}_h)}{\lVert \matr{\tau}_h \rVert} &\geq \beta \lVert \vect{v}_h \rVert_{H^1, h}, & &\forall \vect{v}_h \in U_{h, 0}. \label{eq:inf-sup-b*_h-sym}
    \end{align}
\end{theorem}

\begin{proof}
    It suffices to show that every $\vect{v}_h \in U_{h, 0}$ admits some $\matr{\tau}^s_h \in \Sigma^s_h$ with $b^*_h (\vect{v}_h, \matr{\tau}^s_h) \gtrsim \lVert \vect{v}_h \rVert^2_{H^1, h}$ and $\lVert \matr{\tau}^s_h \rVert \lesssim \lVert \vect{v}_h \rVert_{H^1, h}$. By \eqref{eq:inf-sup-b*_h} there exists $\matr{\tau}_h \in \Sigma_h$ with
    \begin{align*}
        b^*_h (\vect{v}_h, \matr{\tau}_h) \gtrsim \lVert \vect{v}_h \rVert^2_{H^1, h}, \qquad
        \lVert \matr{\tau}_h \rVert \lesssim \lVert \vect{v}_h \rVert_{H^1, h}.
    \end{align*}
    Set $\matr{\tau}^s_h \defeq \mathcal{S} \matr{\tau}_h \in \Sigma^s_h$ by \cref{lem:symmetrization}. The invariance \eqref{eq:symmetrization-2} preserves the first estimate, and the boundedness \eqref{eq:symmetrization-1} preserves the second.
\end{proof}

\subsection{Discrete Korn's Inequality}

For second-order tensor fields we use the quotient norm modulo a scalar trace, $\lVert \matr{\tau} \rVert_{L^2_*} \defeq \inf_{c \in \mathbb{R}} \lVert \matr{\tau} - c \matr{I} \rVert$, whose infimum is attained at $c = \frac{1}{2 \lvert \Omega \rvert} \int_\Omega \tr \matr{\tau}$. We also write $L^2_0 (\Omega) \defeq \{ q \in L^2 (\Omega) : \int_\Omega q = 0 \}$ and $H^1_0 (\Omega; \mathbb{R}^2) \defeq \{ \vect{v} \in H^1 (\Omega; \mathbb{R}^2) : \vect{v} = \vect{0} \text{ on } \partial \Omega \}$.

\begin{lemma}[Discrete Korn's inequality] \label{lem:korn-ineq-Sigma_h}
    For all $\matr{\tau}_h \in \Sigma_h$,
    \begin{align}
        \lVert \matr{\tau}_h \rVert_{L^2_*} &\lesssim \lVert \dev \matr{\tau}_h \rVert + \lVert \matr{\tau}_h \rVert_{\divg, h}. \label{eq:korn-ineq-Sigma_h}
    \end{align}
\end{lemma}

\begin{proof}
    Fix $\matr{\tau}_h \in \Sigma_h$. We set $c \defeq \frac{1}{2 \lvert \Omega \rvert} \int_\Omega \tr \matr{\tau}_h$ and $q_h \defeq \tr \matr{\tau}_h - 2 c$, so that $\matr{\tau}_h - c \matr{I} = \dev \matr{\tau}_h + \tfrac{1}{2} q_h \matr{I}$. By the definition of the quotient norm,
    \begin{align}
        \lVert \matr{\tau}_h \rVert_{L^2_*} = \lVert \matr{\tau}_h - c \matr{I} \rVert \lesssim \lVert \dev \matr{\tau}_h \rVert + \lVert q_h \rVert. \label{prf:korn-ineq-Sigma_h-1}
    \end{align}
    It remains to bound $\lVert q_h \rVert$. Clearly $q_h \in L^2_0 (\Omega)$, and since $\Omega$ is a bounded convex polygon, the divergence operator maps $H^1_0 (\Omega; \mathbb{R}^2)$ onto $L^2_0 (\Omega)$ with a bounded right inverse \cite{girault1986finite}: there exists $\vect{w} \in H^1_0 (\Omega; \mathbb{R}^2)$ with $\divg \vect{w} = q_h$ and $\lVert \vect{w} \rVert_{H^1} \lesssim \lVert q_h \rVert$. Using $\int_\Omega \divg \vect{w} = \int_{\partial \Omega} \vect{w} \cdot \vect{n} = 0$,
    \begin{align*}
        \lVert q_h \rVert^2 = (q_h, \divg \vect{w}) = \left( (\tr \matr{\tau}_h) \matr{I}, \grad \vect{w} \right) = 2 (\matr{\tau}_h, \grad \vect{w}) - 2 (\dev \matr{\tau}_h, \grad \vect{w}).
    \end{align*}
    Integrating by parts elementwise and using $\jump{\matr{\tau}_h \cdot \vect{n}}_{|_e} = \vect{0}$ for $e \in \mathcal{F}^{dl}_h$ together with $\vect{w} = \vect{0}$ on $\partial \Omega$, we obtain
    \begin{align*}
        (\matr{\tau}_h, \grad \vect{w}) = - (\divg \matr{\tau}_h, \vect{w})_{\mathcal{T}_h} + (\jump{\matr{\tau}_h \cdot \vect{n}}, \vect{w})_{\mathcal{F}^{pr, o}_h} \lesssim \lVert \matr{\tau}_h \rVert_{\divg, h} \lVert \vect{w} \rVert_{L^2, h},
    \end{align*}
    while the trace inequality \cite{ern2004theory} gives $\lVert \vect{w} \rVert_{L^2, h} \lesssim \lVert \vect{w} \rVert_{H^1}$. Combining the last two displays yields $\lVert q_h \rVert^2 \lesssim ( \lVert \matr{\tau}_h \rVert_{\divg, h} + \lVert \dev \matr{\tau}_h \rVert ) \lVert q_h \rVert$; dividing by $\lVert q_h \rVert$ and inserting the result into \eqref{prf:korn-ineq-Sigma_h-1} gives \eqref{eq:korn-ineq-Sigma_h}.
\end{proof}

\begin{corollary} \label{cor:L2-control-Sigma_h}
    For all $\matr{\tau}_h \in \Sigma_h$,
    \begin{align}
        \lVert \matr{\tau}_h \rVert^2_{L^2_*} \lesssim (\mathcal{A}^{\frac{1}{2}} \matr{\tau}_h, \matr{\tau}_h - c \matr{I}) + \lVert \matr{\tau}_h \rVert^2_{\divg, h}, \label{eq:L2-control-Sigma_h}
    \end{align}
    where $c \defeq \frac{1}{2 \lvert \Omega \rvert} \int_\Omega \tr \matr{\tau}_h$.
\end{corollary}

\begin{proof}
    With $q_h \defeq \tr \matr{\tau}_h - 2 c \in L^2_0 (\Omega)$ as in the proof of \cref{lem:korn-ineq-Sigma_h}, we have $\matr{\tau}_h - c \matr{I} = \dev \matr{\tau}_h + \tfrac{1}{2} q_h \matr{I}$ and $(\tr \matr{\tau}_h, q_h) = (2 c + q_h, q_h) = \lVert q_h \rVert^2$. Recalling the definition \eqref{eq:A^1/2} of $\mathcal{A}^{1/2}$ and using the $L^2$-orthogonality of $\dev \matr{\tau}_h$ and $\matr{I}$,
    \begin{align*}
        (\mathcal{A}^{\frac{1}{2}} \matr{\tau}_h, \matr{\tau}_h - c \matr{I}) = \frac{1}{\sqrt{2\mu}} \lVert \dev \matr{\tau}_h \rVert^2 + \frac{1}{2 \sqrt{2\lambda + 2\mu}} \lVert q_h \rVert^2 \geq \frac{1}{\sqrt{2\mu}} \lVert \dev \matr{\tau}_h \rVert^2,
    \end{align*}
    which, together with \eqref{eq:korn-ineq-Sigma_h}, yields \eqref{eq:L2-control-Sigma_h}.
\end{proof}

\subsection{Stability}

\begin{theorem}[Stability] \label{thm:stability}
    Let $(\vect{u}_h, \matr{\sigma}_h, \widetilde{\matr{\sigma}}_h, \tens{\chi}_h) \in U_{h, 0} \times \Sigma^s_h \times \widetilde{\Sigma}^s_h \times X^s_h$ solve the discrete scheme \eqref{eq:discrete-scheme-1}--\eqref{eq:discrete-scheme-4}. Then
    \begin{align}
        \lVert \vect{u}_h \rVert + \lVert \vect{u}_h \rVert_{H^1, h} + \lVert \widetilde{\matr{\sigma}}_h \rVert + \lVert \tens{\chi}_h \rVert &\lesssim \lVert \vect{f} \rVert, \label{eq:stability-1} \\
        \lVert \matr{\sigma}_h \rVert_{L^2_*} &\lesssim \left( 1 + \tfrac{\iota}{h} \right) \lVert \vect{f} \rVert. \label{eq:stability-2}
    \end{align}
\end{theorem}

\begin{proof}
    Taking $\matr{\tau}_h = \matr{\sigma}_h$, $\tens{\psi}_h = \tens{\chi}_h$, $\widetilde{\matr{\tau}}_h = \widetilde{\matr{\sigma}}_h$ and $\vect{v}_h = \vect{u}_h$ in \eqref{eq:discrete-scheme-1}--\eqref{eq:discrete-scheme-4} and summing gives
    \begin{align}
        \lVert \widetilde{\matr{\sigma}}_h \rVert^2 + \lVert \tens{\chi}_h \rVert^2 = (\vect{f}, \vect{u}_h) \leq \lVert \vect{f} \rVert \lVert \vect{u}_h \rVert, \label{prf:stability-1}
    \end{align}
    while the inf-sup condition \eqref{eq:inf-sup-b*_h-sym} and \eqref{eq:discrete-scheme-1} give
    \begin{align}
        \lVert \vect{u}_h \rVert_{H^1, h} \lesssim \sup_{\matr{\tau}_h \in \Sigma^s_h \setminus \{ \matr{0} \}} \frac{b^*_h (\vect{u}_h, \matr{\tau}_h)}{\lVert \matr{\tau}_h \rVert} = \sup_{\matr{\tau}_h \in \Sigma^s_h \setminus \{ \matr{0} \}} \frac{(\mathcal{A}^{\frac{1}{2}} \widetilde{\matr{\sigma}}_h, \matr{\tau}_h)}{\lVert \matr{\tau}_h \rVert} \lesssim \lVert \widetilde{\matr{\sigma}}_h \rVert. \label{prf:stability-2}
    \end{align}
    Combining \eqref{prf:stability-1} and \eqref{prf:stability-2} with the discrete Poincar\'e inequality \eqref{eq:poincare-ineq-U_h} yields \eqref{eq:stability-1}.

    For \eqref{eq:stability-2}, let $c \defeq \frac{1}{2 \lvert \Omega \rvert} \int_\Omega \tr \matr{\sigma}_h$, so that \cref{cor:L2-control-Sigma_h} gives
    \begin{align}
        \lVert \matr{\sigma}_h \rVert^2_{L^2_*} \lesssim (\mathcal{A}^{\frac{1}{2}} \matr{\sigma}_h, \matr{\sigma}_h - c \matr{I}) + \lVert \matr{\sigma}_h \rVert^2_{\divg, h}. \label{prf:stability-3}
    \end{align}
    Since $\Sigma^s_h \subset \widetilde{\Sigma}^s_h$, we may take $\widetilde{\matr{\tau}}_h = \matr{\sigma}_h - c \matr{I}$ in \eqref{eq:discrete-scheme-3}; together with the inverse inequality \eqref{eq:inverse-ineq-X_h} this bounds the first term,
    \begin{align}
        (\mathcal{A}^{\frac{1}{2}} \matr{\sigma}_h, \matr{\sigma}_h - c \matr{I}) = (\widetilde{\matr{\sigma}}_h + \iota \divg \tens{\chi}_h, \matr{\sigma}_h - c \matr{I}) \lesssim \left( \lVert \widetilde{\matr{\sigma}}_h \rVert + \tfrac{\iota}{h} \lVert \tens{\chi}_h \rVert \right) \lVert \matr{\sigma}_h \rVert_{L^2_*}, \label{prf:stability-4}
    \end{align}
    while the inf-sup condition \eqref{eq:inf-sup-b_h} and \eqref{eq:discrete-scheme-4} bound the second,
    \begin{align}
        \lVert \matr{\sigma}_h \rVert_{\divg, h} \lesssim \sup_{\vect{v}_h \in U_{h, 0} \setminus \{ \vect{0} \}} \frac{b_h (\matr{\sigma}_h, \vect{v}_h)}{\lVert \vect{v}_h \rVert} = \sup_{\vect{v}_h \in U_{h, 0} \setminus \{ \vect{0} \}} \frac{(\vect{f}, \vect{v}_h)}{\lVert \vect{v}_h \rVert} \lesssim \lVert \vect{f} \rVert. \label{prf:stability-5}
    \end{align}
    Substituting \eqref{prf:stability-4} and \eqref{prf:stability-5} into \eqref{prf:stability-3} and invoking \eqref{eq:stability-1} gives \eqref{eq:stability-2}.
\end{proof}

\begin{remark}
    Compared with \eqref{eq:stability-1}, the stability estimate \eqref{eq:stability-2} for the total stress carries the extra factor $\iota/h$; since $\iota$ is typically a small parameter in applications, this factor remains under control and the estimate is still meaningful in practice.
\end{remark}

\begin{remark} \label{rmk:incompressible-limit}
    If $\lambda$ is bounded above, i.e., $\lambda \lesssim 1$, then taking $\widetilde{\matr{\tau}}_h = \matr{\sigma}_h$ in \eqref{eq:discrete-scheme-3} and arguing as in the proof of \cref{thm:stability} improves the stability estimate \eqref{eq:stability-2} for the total stress to the $L^2$ norm. If $\lambda$ is not bounded above, however, the stability of the total stress degenerates in the direction of the constant spherical modes $c \matr{I}$, $c \in \mathbb{R}$, as it does in the Hellinger--Reissner formulation of linear elasticity. Indeed, as $\lambda \to \infty$ the material approaches the incompressible limit and the total stress is determined only up to such a mode, so that an additional normalization, for instance $\int_\Omega \tr \matr{\sigma}_h = 0$, is required to fix it.
\end{remark}

\section{Error Analysis} \label{sec:5}

\subsection{Preliminaries}

Let $\Pi^U_h$ and $\Pi^{\Sigma}_h$ be the canonical projections onto the SDG spaces $U_h$ and $\Sigma_h$, defined by preserving the degrees of freedom of $U_h$ and $\Sigma_h$. In view of the definitions \eqref{eq:b_h}--\eqref{eq:b*_h}, these projections give the orthogonality properties
\begin{align}
    b_h (\matr{\sigma} - \Pi^{\Sigma}_h \matr{\sigma}, \vect{v}_h) &= 0, & &\forall \vect{v}_h \in U_{h, 0}, \label{eq:orthogonality-Pi-Sigma} \\
    b^*_h (\vect{u} - \Pi^U_h \vect{u}, \matr{\tau}_h) &= 0, & &\forall \matr{\tau}_h \in \Sigma_h. \label{eq:orthogonality-Pi-U}
\end{align}
The corresponding projection estimates are given in \cite{chung2009optimal, kim2013staggered}, and the low-regularity estimate for $\Pi^{\Sigma}_h$ in \cite{zhao2018staggered}:
\begin{align}
    \lVert \vect{u} - \Pi^U_h \vect{u} \rVert + h \lVert \vect{u} - \Pi^U_h \vect{u} \rVert_{H^1, h} &\lesssim h^r \lVert \vect{u} \rVert_r, & &1 \leq r \leq k + 1, \label{eq:proj-err-Pi-U} \\
    \lVert \matr{\sigma} - \Pi^{\Sigma}_h \matr{\sigma} \rVert &\lesssim h^r \lVert \matr{\sigma} \rVert_r, & &1 \leq r \leq k + 1, \label{eq:proj-est-Pi-Sigma} \\
    \lVert \matr{\sigma} - \Pi^{\Sigma}_h \matr{\sigma} \rVert &\lesssim h^\epsilon (\lVert \matr{\sigma} \rVert_\epsilon + \lVert \divg \matr{\sigma} \rVert), & &0 < \epsilon \leq 1. \label{eq:proj-est-Pi-Sigma-low-reg}
\end{align}

Let $\Pi^{RT}_h$ denote the Raviart--Thomas projection \cite{boffi2013mixed} onto $RT_k (\mathcal{T}_h)$ and $\pi^{DG}_h$ the $L^2$ projection onto $DG_k (\mathcal{T}_h)$. We define $\Pi^{X, s}_h$ and $\pi^{\widetilde{\Sigma}, s}_h$, the projection onto $X^s_h$ and $\widetilde{\Sigma}^s_h$, by applying these componentwise, $(\Pi^{X, s}_h \tens{\chi})_{ij\bullet} \defeq \Pi^{RT}_h \tens{\chi}_{ij\bullet}$ and $(\pi^{\widetilde{\Sigma}, s}_h \widetilde{\matr{\sigma}})_{ij} \defeq \pi^{DG}_h \widetilde{\matr{\sigma}}_{ij}$, so that the approximation properties,
\begin{align}
    \lVert \tens{\chi} - \Pi^{X, s}_h \tens{\chi} \rVert + \lVert \divg (\tens{\chi} - \Pi^{X, s}_h \tens{\chi}) \rVert &\lesssim h^r \lVert \tens{\chi} \rVert_r, & &1 \leq r \leq k + 1, \label{eq:proj-est-Pi-X-s} \\
    \lVert \widetilde{\matr{\sigma}} - \pi^{\widetilde{\Sigma}, s}_h \widetilde{\matr{\sigma}} \rVert &\lesssim h^r \lVert \widetilde{\matr{\sigma}} \rVert_r, & &0 \leq r \leq k + 1, \label{eq:proj-est-pi-tilde-Sigma-s}
\end{align}
and the commuting diagram property,
\begin{align}
    (\divg (\tens{\chi} - \Pi^{X, s}_h \tens{\chi}), \widetilde{\matr{\tau}}_h) &= 0, & &\forall \widetilde{\matr{\tau}}_h \in \widetilde{\Sigma}^s_h, \label{eq:orthogonality-Pi-X-s} \\
    (\pi^{\widetilde{\Sigma}, s}_h \widetilde{\matr{\sigma}} - \widetilde{\matr{\sigma}}, \divg \tens{\psi}_h) &= 0, & &\forall \tens{\psi}_h \in X^s_h, \label{eq:orthogonality-pi-tilde-Sigma-s}
\end{align}
are then inherited directly from those of $\Pi^{RT}_h$ and $\pi^{DG}_h$.

\subsection{Optimal Convergence}

Let $\mathcal{S}$ be the symmetrization operator of \cref{lem:symmetrization}. The projection $\Pi^{\Sigma, s}_h \defeq \mathcal{S} \circ \Pi^{\Sigma}_h$ onto $\Sigma^s_h$ inherits both the approximation and the orthogonality property of $\Pi^{\Sigma}_h$ as shown in the following lemma.

\begin{lemma} \label{lem:proj-est-Pi-Sigma-s}
    For any symmetric $\matr{\sigma} \in H^r (\Omega; \mathbb{R}^{2 \times 2})$ with $1 \leq r \leq k + 1$,
    \begin{align}
        \lVert \matr{\sigma} - \Pi^{\Sigma, s}_h \matr{\sigma} \rVert &\lesssim h^r \lVert \matr{\sigma} \rVert_r, \label{eq:proj-est-Pi-Sigma-s} \\
        b_h (\matr{\sigma} - \Pi^{\Sigma, s}_h \matr{\sigma}, \vect{v}_h) &= 0, & &\forall \vect{v}_h \in U_{h, 0}. \label{eq:orthogonality-Pi-Sigma-s}
    \end{align}
\end{lemma}

\begin{proof}
    By the construction of $\mathcal{S}$ in the proof of \cref{lem:symmetrization},
    \begin{align}
        \Pi^{\Sigma, s}_h \matr{\sigma} = \Pi^{\Sigma}_h \matr{\sigma} + \curl \vect{w}_h \label{prf:consistency-Pi-Sigma-s-1}
    \end{align}
    for some $\vect{w}_h \in H^1 (\Omega; \mathbb{R}^2)$ with $\lVert \vect{w}_h \rVert_{H^1} \lesssim \lVert (\Pi^{\Sigma}_h \matr{\sigma})_{21} - (\Pi^{\Sigma}_h \matr{\sigma})_{12} \rVert$ and $b^*_h (\vect{v}_h, \curl \vect{w}_h) = 0$ for all $\vect{v}_h \in U_{h, 0}$.

    For \eqref{eq:proj-est-Pi-Sigma-s}, the symmetry of $\matr{\sigma}$ gives
    \begin{align*}
        \lVert \vect{w}_h \rVert_{H^1} \lesssim \lVert (\Pi^{\Sigma}_h \matr{\sigma} - \matr{\sigma})_{21} - (\Pi^{\Sigma}_h \matr{\sigma} - \matr{\sigma})_{12} \rVert \lesssim \lVert \matr{\sigma} - \Pi^{\Sigma}_h \matr{\sigma} \rVert,
    \end{align*}
    so that, by \eqref{prf:consistency-Pi-Sigma-s-1},
    \begin{align*}
        \lVert \matr{\sigma} - \Pi^{\Sigma, s}_h \matr{\sigma} \rVert \leq \lVert \matr{\sigma} - \Pi^{\Sigma}_h \matr{\sigma} \rVert + \lVert \curl \vect{w}_h \rVert \lesssim \lVert \matr{\sigma} - \Pi^{\Sigma}_h \matr{\sigma} \rVert,
    \end{align*}
    which combined with \eqref{eq:proj-est-Pi-Sigma} yields \eqref{eq:proj-est-Pi-Sigma-s}. For \eqref{eq:orthogonality-Pi-Sigma-s}, \eqref{prf:consistency-Pi-Sigma-s-1} gives $b_h (\matr{\sigma} - \Pi^{\Sigma, s}_h \matr{\sigma}, \vect{v}_h) = b_h (\matr{\sigma} - \Pi^{\Sigma}_h \matr{\sigma}, \vect{v}_h) - b_h (\curl \vect{w}_h, \vect{v}_h)$; the first term vanishes by \eqref{eq:orthogonality-Pi-Sigma}, and the second equals $b^*_h (\vect{v}_h, \curl \vect{w}_h) = 0$ by the adjointness \eqref{eq:adjointness}.
\end{proof}

Expressing $\matr{\sigma}$, $\widetilde{\matr{\sigma}}$ and $\tens{\chi}$ in terms of $\vect{u}$,
\begin{align}
    \matr{\sigma} &= (1 - \iota^2 \Delta) \left( 2\mu \symgrad \vect{u} + \lambda (\divg \vect{u}) \matr{I} \right), \label{eq:u-form-of-sigma} \\
    \widetilde{\matr{\sigma}} &= \sqrt{2\mu} \symgrad \vect{u} + \frac{1}{2} \left( \sqrt{2\lambda + 2\mu} - \sqrt{2\mu} \right) (\divg \vect{u}) \matr{I}, \label{eq:u-form-of-tilde-sigma} \\
    \tens{\chi} &= \iota \grad \widetilde{\matr{\sigma}}, \label{eq:u-form-of-chi}
\end{align}
and introducing, for $m \geq 0$, the $\lambda$-dependent norm $\tnorm{\vect{v}}^2_{m + 1} \defeq \lVert \vect{v} \rVert^2_{m + 1} + \lambda \lVert \divg \vect{v} \rVert^2_m$ together with the shorthand
\begin{align}
    \mathcal{E}_r (\vect{u}) \defeq \tnorm{\vect{u}}_{r + 1} + \iota \tnorm{\vect{u}}_{r + 2} + \iota^2 \tnorm{\vect{u}}_{r + 3}, \label{eq:E_r}
\end{align}
the projection estimates \eqref{eq:proj-est-Pi-X-s}, \eqref{eq:proj-est-pi-tilde-Sigma-s} and \eqref{eq:proj-est-Pi-Sigma-s} take the form
\begin{align}
    \lVert \matr{\sigma} - \Pi^{\Sigma, s}_h \matr{\sigma} \rVert &\lesssim h^r \mathcal{E}_r (\vect{u}), & &1 \leq r \leq k + 1, \label{eq:proj-err-Pi-Sigma-s} \\
    \lVert \widetilde{\matr{\sigma}} - \pi^{\widetilde{\Sigma}, s}_h \widetilde{\matr{\sigma}} \rVert &\lesssim h^r \tnorm{\vect{u}}_{r + 1}, & &0 \leq r \leq k + 1, \label{eq:proj-err-pi-tilde-Sigma-s} \\
    \lVert \tens{\chi} - \Pi^{X, s}_h \tens{\chi} \rVert &\lesssim h^r \iota \tnorm{\vect{u}}_{r + 2}, & &1 \leq r \leq k + 1. \label{eq:proj-err-Pi-X-s}
\end{align}

\begin{theorem}[Optimal convergence] \label{thm:optimal-convergence}
    Let $(\vect{u}, \matr{\sigma}, \widetilde{\matr{\sigma}}, \tens{\chi})$ solve the SGE problem \eqref{eq:first-order-SGE-1}--\eqref{eq:first-order-SGE-4} and let $(\vect{u}_h, \matr{\sigma}_h, \widetilde{\matr{\sigma}}_h, \tens{\chi}_h)$ solve the discrete scheme \eqref{eq:discrete-scheme-1}--\eqref{eq:discrete-scheme-4}. Under sufficient regularity, the following estimates hold for $1 \leq r \leq k + 1$:
    \begin{align}
        \lVert \vect{u} - \vect{u}_h \rVert + h \lVert \vect{u} - \vect{u}_h \rVert_{H^1, h} + \lVert \widetilde{\matr{\sigma}} - \widetilde{\matr{\sigma}}_h \rVert + \lVert \tens{\chi} - \tens{\chi}_h \rVert &\lesssim h^r \mathcal{E}_r (\vect{u}), \label{eq:optimal-convergence-1} \\
        \lVert \matr{\sigma} - \matr{\sigma}_h \rVert_{L^2_*} &\lesssim (h^r + \iota h^{r - 1}) \mathcal{E}_r (\vect{u}). \label{eq:optimal-convergence-2}
    \end{align}
\end{theorem}

\begin{proof}
    Subtracting the discrete scheme \eqref{eq:discrete-scheme-1}--\eqref{eq:discrete-scheme-4} from the variational formulation \eqref{eq:variation-formulation-1}--\eqref{eq:variation-formulation-4} and using the orthogonality properties \eqref{eq:orthogonality-Pi-U}, \eqref{eq:orthogonality-Pi-Sigma-s}, \eqref{eq:orthogonality-pi-tilde-Sigma-s} and \eqref{eq:orthogonality-Pi-X-s} of $\Pi^U_h$, $\Pi^{\Sigma, s}_h$, $\pi^{\widetilde{\Sigma}, s}_h$ and $\Pi^{X, s}_h$, we obtain the error equations
    \begin{align}
        &(\mathcal{A}^{\frac{1}{2}} (\pi^{\widetilde{\Sigma}, s}_h \widetilde{\matr{\sigma}} - \widetilde{\matr{\sigma}}_h), \matr{\tau}_h) - b^*_h (\Pi^U_h \vect{u} - \vect{u}_h, \matr{\tau}_h) = 0, \quad \forall \matr{\tau}_h \in \Sigma^s_h, \label{eq:error-eq-1} \\
        &\begin{aligned} \label{eq:error-eq-2}
            &(\Pi^{X, s}_h \tens{\chi} - \tens{\chi}_h, \tens{\psi}_h) + \iota (\pi^{\widetilde{\Sigma}, s}_h \widetilde{\matr{\sigma}} - \widetilde{\matr{\sigma}}_h, \divg \tens{\psi}_h) \\ &\qquad = -(\tens{\chi} - \Pi^{X, s}_h \tens{\chi}, \tens{\psi}_h), \quad \forall \tens{\psi}_h \in X^s_h,
        \end{aligned} \\
        &\begin{aligned} \label{eq:error-eq-3}
            &(\pi^{\widetilde{\Sigma}, s}_h \widetilde{\matr{\sigma}} - \widetilde{\matr{\sigma}}_h, \widetilde{\matr{\tau}}_h) - (\mathcal{A}^{\frac{1}{2}} (\Pi^{\Sigma, s}_h \matr{\sigma} - \matr{\sigma}_h), \widetilde{\matr{\tau}}_h) - \iota (\divg (\Pi^{X, s}_h \tens{\chi} - \tens{\chi}_h), \widetilde{\matr{\tau}}_h) \\ &\qquad = (\mathcal{A}^{\frac{1}{2}} (\matr{\sigma} - \Pi^{\Sigma, s}_h \matr{\sigma}), \widetilde{\matr{\tau}}_h), \quad \forall \widetilde{\matr{\tau}}_h \in \widetilde{\Sigma}^s_h,
        \end{aligned} \\
        &b_h (\Pi^{\Sigma, s}_h \matr{\sigma} - \matr{\sigma}_h, \vect{v}_h) = 0, \quad \forall \vect{v}_h \in U_{h, 0}. \label{eq:error-eq-4}
    \end{align}

    \emph{Proof of \eqref{eq:optimal-convergence-1}.} Taking $\matr{\tau}_h = \Pi^{\Sigma, s}_h \matr{\sigma} - \matr{\sigma}_h$, $\tens{\psi}_h = \Pi^{X, s}_h \tens{\chi} - \tens{\chi}_h$, $\widetilde{\matr{\tau}}_h = \pi^{\widetilde{\Sigma}, s}_h \widetilde{\matr{\sigma}} - \widetilde{\matr{\sigma}}_h$ and $\vect{v}_h = \Pi^U_h \vect{u} - \vect{u}_h$ in \eqref{eq:error-eq-1}--\eqref{eq:error-eq-4} and summing, we obtain
    \begin{align*}
        &\lVert \pi^{\widetilde{\Sigma}, s}_h \widetilde{\matr{\sigma}} - \widetilde{\matr{\sigma}}_h \rVert^2 + \lVert \Pi^{X, s}_h \tens{\chi} - \tens{\chi}_h \rVert^2 \\
        &= - (\tens{\chi} - \Pi^{X, s}_h \tens{\chi}, \Pi^{X, s}_h \tens{\chi} - \tens{\chi}_h) + (\mathcal{A}^{\frac{1}{2}} (\matr{\sigma} - \Pi^{\Sigma, s}_h \matr{\sigma}), \pi^{\widetilde{\Sigma}, s}_h \widetilde{\matr{\sigma}} - \widetilde{\matr{\sigma}}_h) \\
        &\lesssim \lVert \tens{\chi} - \Pi^{X, s}_h \tens{\chi} \rVert \lVert \Pi^{X, s}_h \tens{\chi} - \tens{\chi}_h \rVert + \lVert \matr{\sigma} - \Pi^{\Sigma, s}_h \matr{\sigma} \rVert \lVert \pi^{\widetilde{\Sigma}, s}_h \widetilde{\matr{\sigma}} - \widetilde{\matr{\sigma}}_h \rVert,
    \end{align*}
    which, with the projection estimates \eqref{eq:proj-err-Pi-X-s} and \eqref{eq:proj-err-Pi-Sigma-s}, gives
    \begin{align}
        \lVert \pi^{\widetilde{\Sigma}, s}_h \widetilde{\matr{\sigma}} - \widetilde{\matr{\sigma}}_h \rVert + \lVert \Pi^{X, s}_h \tens{\chi} - \tens{\chi}_h \rVert \lesssim h^r \mathcal{E}_r (\vect{u}). \label{prf:optimal-convergence-1}
    \end{align}
    By the inf-sup condition \eqref{eq:inf-sup-b*_h-sym} and the error equation \eqref{eq:error-eq-1},
    \begin{align*}
        \lVert \Pi^U_h \vect{u} - \vect{u}_h \rVert_{H^1, h}
        \lesssim \sup_{\matr{\tau}_h \in \Sigma^s_h \setminus \{ \matr{0} \}} \frac{ (\mathcal{A}^{\frac{1}{2}} (\pi^{\widetilde{\Sigma}, s}_h \widetilde{\matr{\sigma}} - \widetilde{\matr{\sigma}}_h), \matr{\tau}_h) }{\lVert \matr{\tau}_h \rVert} \lesssim \lVert \pi^{\widetilde{\Sigma}, s}_h \widetilde{\matr{\sigma}} - \widetilde{\matr{\sigma}}_h \rVert,
    \end{align*}
    which, with the discrete Poincar\'e inequality \eqref{eq:poincare-ineq-U_h}, gives
    \begin{align}
        \lVert \Pi^U_h \vect{u} - \vect{u}_h \rVert + \lVert \Pi^U_h \vect{u} - \vect{u}_h \rVert_{H^1, h} \lesssim \lVert \pi^{\widetilde{\Sigma}, s}_h \widetilde{\matr{\sigma}} - \widetilde{\matr{\sigma}}_h \rVert. \label{prf:optimal-convergence-2}
    \end{align}
    Combining \eqref{prf:optimal-convergence-1} and \eqref{prf:optimal-convergence-2} with the projection estimates \eqref{eq:proj-err-Pi-U}, \eqref{eq:proj-err-pi-tilde-Sigma-s} and \eqref{eq:proj-err-Pi-X-s} yields \eqref{eq:optimal-convergence-1}.

    \emph{Proof of \eqref{eq:optimal-convergence-2}.} The inf-sup condition \eqref{eq:inf-sup-b_h} and the error equation \eqref{eq:error-eq-4} give $\lVert \Pi^{\Sigma, s}_h \matr{\sigma} - \matr{\sigma}_h \rVert_{\divg, h} = 0$, so that \cref{cor:L2-control-Sigma_h} reduces to
    \begin{align}
        \lVert \Pi^{\Sigma, s}_h \matr{\sigma} - \matr{\sigma}_h \rVert_{L^2_*}^2 \lesssim (\mathcal{A}^{\frac{1}{2}} (\Pi^{\Sigma, s}_h \matr{\sigma} - \matr{\sigma}_h), \Pi^{\Sigma, s}_h \matr{\sigma} - \matr{\sigma}_h - c \matr{I}), \label{prf:optimal-convergence-3}
    \end{align}
    where $c \defeq \frac{1}{2\lvert \Omega \rvert} \int_{\Omega} \tr (\Pi^{\Sigma, s}_h \matr{\sigma} - \matr{\sigma}_h)$. Taking $\widetilde{\matr{\tau}}_h = \Pi^{\Sigma, s}_h \matr{\sigma} - \matr{\sigma}_h - c \matr{I}$ in \eqref{eq:error-eq-3} bounds the right-hand side by
    \begin{align*}
        \big( \lVert \pi^{\widetilde{\Sigma}, s}_h \widetilde{\matr{\sigma}} - \widetilde{\matr{\sigma}}_h \rVert + \iota \lVert \divg (\Pi^{X, s}_h \tens{\chi} - \tens{\chi}_h) \rVert + \lVert \matr{\sigma} - \Pi^{\Sigma, s}_h \matr{\sigma} \rVert \big) \lVert \Pi^{\Sigma, s}_h \matr{\sigma} - \matr{\sigma}_h \rVert_{L^2_*},
    \end{align*}
    so that, by the inverse inequality \eqref{eq:inverse-ineq-X_h},
    \begin{align*}
        \lVert \Pi^{\Sigma, s}_h \matr{\sigma} - \matr{\sigma}_h \rVert_{L^2_*} \lesssim \lVert \pi^{\widetilde{\Sigma}, s}_h \widetilde{\matr{\sigma}} - \widetilde{\matr{\sigma}}_h \rVert + \tfrac{\iota}{h} \lVert \Pi^{X, s}_h \tens{\chi} - \tens{\chi}_h \rVert + \lVert \matr{\sigma} - \Pi^{\Sigma, s}_h \matr{\sigma} \rVert.
    \end{align*}
    Together with \eqref{prf:optimal-convergence-1} and the projection estimate \eqref{eq:proj-err-Pi-Sigma-s}, this implies \eqref{eq:optimal-convergence-2}.
\end{proof}

\begin{remark}
    It is instructive to compare these estimates with those of the displacement-based methods, in particular the $H^2$-nonconforming elements of \cite{chen2023robust, liao2023taylorhood, chen2025nonconforming}. For a displacement space of degree~$k$, those methods satisfy
    \begin{align}
        \lVert \vect{u} - \vect{u}_h \rVert_{1, h} + \iota \lVert \vect{u} - \vect{u}_h \rVert_{2, h} \lesssim (h^r + \iota h^{r - 1}) \tnorm{\vect{u}}_{r + 1}, \qquad 1 \leq r \leq k, \label{eq:nonconforming-optimal-convergence}
    \end{align}
    where $\lVert \cdot \rVert_{1, h}$ and $\lVert \cdot \rVert_{2, h}$ denote discrete $H^1$ and $H^2$ norms. In the discrete $H^1$ norm they are optimal only when $\iota \lesssim h$ and lose one order otherwise, whereas our method attains the optimal rate for the displacement uniformly in $\iota$. Our method also approximates the total stress directly, with the same convergence pattern as \eqref{eq:nonconforming-optimal-convergence}; the displacement-based theory does not cover such an approximation, since the total stress involves third-order derivatives of the displacement. The price is a higher regularity requirement, a trade-off inherent in the mixed formulation.
\end{remark}

\begin{remark}
    \cref{thm:optimal-convergence} is developed for the regime in which $\iota$ is not too small. It tracks explicitly the dependence of the error on the regularity of $\vect{u}$ and on the two parameters $\lambda$ and $\iota$; since $\lambda$ enters the norm $\tnorm{\cdot}$ only in combination with $\divg \vect{u}$, the estimates are already robust with respect to $\lambda$. As $\iota \to 0$, however, the solution develops a boundary layer and its regularity degenerates correspondingly, so that the estimates cease to be informative. This motivates the parameter-robust analysis of the next subsection, which bounds the error directly in terms of $\vect{f}$.
\end{remark}

\subsection{Parameter Robustness}

Regularity theory for classical elasticity \cite{brenner1992linear} gives $\tnorm{\vect{u}_0}_2 \lesssim \lVert \vect{f} \rVert$, and that for strain gradient elasticity \cite{liao2023taylorhood, chen2023robust} gives
\begin{align}
    \iota \tnorm{\vect{u}}_2 + \iota^2 \tnorm{\vect{u}}_3 + \tnorm{\vect{u} - \vect{u}_0}_1 &\lesssim \iota^{1/2} \lVert \vect{f} \rVert. \label{eq:regularity}
\end{align}
Extrapolating from these results, we assume the following.

\begin{hypothesis} \label{hyp:regularity}
    The solution $\vect{u}_0$ of the classical elasticity problem \eqref{eq:classical-elasticity} and the solution $\vect{u}$ of the SGE problem \eqref{eq:SGE} satisfy
    \begin{align}
        \tnorm{\vect{u}_0}_{2 + r} &\lesssim \lVert \vect{f} \rVert_r, & &r \geq 0, \label{eq:hyp-regularity-u_0} \\
        \iota^{1 + r} \tnorm{\vect{u}}_{2 + r} &\lesssim
        \left\{ \begin{aligned}
            &\iota^{1/2} \lVert \vect{f} \rVert, \\
            &\iota^{1/2} \lVert \vect{f} \rVert_{r - 1},
        \end{aligned} \right.
        & & \begin{aligned}
            &0 \leq r \leq 1, \\
            &r \geq 1,
        \end{aligned} \label{eq:hyp-regularity-u} \\
        \iota^r \tnorm{\vect{u} - \vect{u}_0}_{1 + r} &\lesssim \iota^{1/2} \lVert \vect{f} \rVert, & &0 \leq r \leq 1. \label{eq:hyp-regularity-u---u_0}
    \end{align}
\end{hypothesis}

\begin{lemma} \label{lem:proj-err-low-reg}
    Under \cref{hyp:regularity}, for $0 < \epsilon \leq 1$,
    \begin{align}
        \lVert \vect{u} - \Pi^U_h \vect{u} \rVert &\lesssim h^{1 + \epsilon} \iota^{1/2 - \epsilon} \lVert \vect{f} \rVert + h^{r + 1} \lVert \vect{f} \rVert_r, & &0 \leq r \leq k, \label{eq:proj-err-Pi-U-low-reg} \\
        \lVert \vect{u} - \Pi^U_h \vect{u} \rVert_{H^1, h} &\lesssim h^{\epsilon} \iota^{1/2 - \epsilon} \lVert \vect{f} \rVert + h^{r + 1} \lVert \vect{f} \rVert_r, & &0 \leq r \leq k - 1, \label{eq:proj-err-Pi-U-H^1-low-reg} \\
        \lVert \widetilde{\matr{\sigma}} - \pi^{\widetilde{\Sigma}, s}_h \widetilde{\matr{\sigma}} \rVert &\lesssim h^\epsilon \iota^{1/2 - \epsilon} \lVert \vect{f} \rVert + h^{r + 1} \lVert \vect{f} \rVert_r, & &0 \leq r \leq k, \label{eq:proj-err-pi-tilde-Sigma-s-low-reg} \\
        \lVert \matr{\sigma} - \Pi^{\Sigma, s}_h \matr{\sigma} \rVert &\lesssim h^\epsilon \iota^{1/2 - \epsilon} \lVert \vect{f} \rVert_{\epsilon} + h^{r + 1} \lVert \vect{f} \rVert_r, & &0 \leq r \leq k. \label{eq:proj-err-Pi-Sigma-s-low-reg}
    \end{align}
\end{lemma}

\begin{proof}
    We prove only \eqref{eq:proj-err-Pi-Sigma-s-low-reg}; the remaining three estimates are obtained similarly. Let $\matr{\sigma}_0 \defeq 2\mu \symgrad \vect{u}_0 + \lambda (\divg \vect{u}_0) \matr{I}$ and split
    \begin{align*}
        \matr{\sigma} - \Pi^{\Sigma, s}_h \matr{\sigma} = \big[ (\matr{\sigma} - \matr{\sigma}_0) - \Pi^{\Sigma, s}_h (\matr{\sigma} - \matr{\sigma}_0) \big] + \big( \matr{\sigma}_0 - \Pi^{\Sigma, s}_h \matr{\sigma}_0 \big).
    \end{align*}
    Since the proof of \cref{lem:proj-est-Pi-Sigma-s} gives $\lVert \matr{\sigma} - \Pi^{\Sigma, s}_h \matr{\sigma} \rVert \lesssim \lVert \matr{\sigma} - \Pi^{\Sigma}_h \matr{\sigma} \rVert$, the projection estimates \eqref{eq:proj-est-Pi-Sigma} and \eqref{eq:proj-est-Pi-Sigma-low-reg} apply to $\Pi^{\Sigma, s}_h$ as well, whence
    \begin{align*}
        \lVert \matr{\sigma} - \Pi^{\Sigma, s}_h \matr{\sigma} \rVert \lesssim h^\epsilon \lVert \matr{\sigma} - \matr{\sigma}_0 \rVert_\epsilon + h^{r + 1} \lVert \matr{\sigma}_0 \rVert_{r + 1},
    \end{align*}
    the divergence term in \eqref{eq:proj-est-Pi-Sigma-low-reg} having dropped out because $\divg \matr{\sigma} = \divg \matr{\sigma}_0 = -\vect{f}$. By the displacement form \eqref{eq:u-form-of-sigma},
    \begin{align}
        \matr{\sigma} - \matr{\sigma}_0 = 2\mu \symgrad (\vect{u} - \vect{u}_0) + \lambda \left( \divg (\vect{u} - \vect{u}_0) \right) \matr{I} - \iota^2 \Delta \left( 2\mu \symgrad \vect{u} + \lambda (\divg \vect{u}) \matr{I} \right), \label{prf:lem:proj-err-low-reg}
    \end{align}
    so that
    \begin{align*}
        \lVert \matr{\sigma} - \Pi^{\Sigma, s}_h \matr{\sigma} \rVert \lesssim h^\epsilon \left( \tnorm{\vect{u} - \vect{u}_0}_{1 + \epsilon} + \iota^2 \tnorm{\vect{u}}_{3 + \epsilon} \right) + h^{r + 1} \tnorm{\vect{u}_0}_{r + 2}.
    \end{align*}
    The regularity hypotheses \eqref{eq:hyp-regularity-u_0}--\eqref{eq:hyp-regularity-u---u_0} then give \eqref{eq:proj-err-Pi-Sigma-s-low-reg}.
\end{proof}

In place of the Raviart--Thomas projection $\Pi^{X, s}_h$ used above, we now employ the $L^2$ projection $\pi^{X, s}_h$ onto $X^s_h$, for which standard approximation theory \cite{ern2004theory} gives $\lVert \tens{\chi} - \pi^{X, s}_h \tens{\chi} \rVert \lesssim h^r \lVert \tens{\chi} \rVert_r$ for $0 \leq r \leq k + 1$ and $\lVert \divg (\tens{\chi} - \pi^{X, s}_h \tens{\chi}) \rVert \lesssim h^{r - 1} \lVert \tens{\chi} \rVert_r$ for $1 \leq r \leq k + 1$. Taking $r = \epsilon$ in the first and $r = 1 + \epsilon$ in the second and using the regularity hypothesis \eqref{eq:hyp-regularity-u},
\begin{align}
    \lVert \tens{\chi} - \pi^{X, s}_h \tens{\chi} \rVert &\lesssim h^{\epsilon} \iota \tnorm{\vect{u}}_{2 + \epsilon} \lesssim h^{\epsilon} \iota^{1/2 - \epsilon} \lVert \vect{f} \rVert, \label{eq:proj-err-pi-X-s-low-reg} \\
    \iota \lVert \divg (\tens{\chi} - \pi^{X, s}_h \tens{\chi}) \rVert &\lesssim h^{\epsilon} \iota^2 \tnorm{\vect{u}}_{3 + \epsilon} \lesssim h^{\epsilon} \iota^{1/2 - \epsilon} \lVert \vect{f} \rVert_{\epsilon}. \label{eq:proj-err-pi-X-s-low-reg-div}
\end{align}

\begin{theorem}[Parameter robustness] \label{thm:parameter-robust}
    Let $(\vect{u}, \matr{\sigma}, \widetilde{\matr{\sigma}}, \tens{\chi})$ solve the SGE problem \eqref{eq:first-order-SGE-1}--\eqref{eq:first-order-SGE-4} and let $(\vect{u}_h, \matr{\sigma}_h, \widetilde{\matr{\sigma}}_h, \tens{\chi}_h)$ solve the discrete scheme \eqref{eq:discrete-scheme-1}--\eqref{eq:discrete-scheme-4}. Under \cref{hyp:regularity}, the following estimates hold for $0 < \epsilon \leq 1/2$, with $0 \leq r \leq k$ in \eqref{eq:parameter-robust-1}, $0 \leq r \leq k - 1$ in \eqref{eq:parameter-robust-2} and $1 \leq r \leq k$ in \eqref{eq:parameter-robust-3}:
    \begin{align}
        \lVert \vect{u} - \vect{u}_h \rVert + \lVert \widetilde{\matr{\sigma}} - \widetilde{\matr{\sigma}}_h \rVert + \lVert \tens{\chi} - \tens{\chi}_h \rVert &\lesssim h^\epsilon \iota^{1/2 - \epsilon} \lVert \vect{f} \rVert_{\epsilon} + h^{r + 1} \lVert \vect{f} \rVert_r, \label{eq:parameter-robust-1} \\
        \lVert \vect{u} - \vect{u}_h \rVert_{H^1, h} &\lesssim h^\epsilon \iota^{1/2 - \epsilon} \lVert \vect{f} \rVert_{\epsilon} + h^{r + 1} \lVert \vect{f} \rVert_r, \label{eq:parameter-robust-2} \\
        \lVert \matr{\sigma} - \matr{\sigma}_h \rVert_{L^2_*} &\lesssim \big( \iota^{1/2} + h^\epsilon \iota^{1/2 - \epsilon} \big) \lVert \vect{f} \rVert_1 + (h^r + \iota h^{r - 1}) \lVert \vect{f} \rVert_r. \label{eq:parameter-robust-3}
    \end{align}
\end{theorem}

\begin{proof}
    Proceeding as in the proof of \cref{thm:optimal-convergence}, but with $\Pi^{X, s}_h$ replaced by $\pi^{X, s}_h$ and using its $L^2$-orthogonality, we obtain the error equations
    \begin{align}
        &(\mathcal{A}^{\frac{1}{2}} (\pi^{\widetilde{\Sigma}, s}_h \widetilde{\matr{\sigma}} - \widetilde{\matr{\sigma}}_h), \matr{\tau}_h) - b^*_h (\Pi^U_h \vect{u} - \vect{u}_h, \matr{\tau}_h) = 0, \quad \forall \matr{\tau}_h \in \Sigma^s_h, \label{eq:error-eq-new-1} \\
        &(\pi^{X, s}_h \tens{\chi} - \tens{\chi}_h, \tens{\psi}_h) + \iota (\pi^{\widetilde{\Sigma}, s}_h \widetilde{\matr{\sigma}} - \widetilde{\matr{\sigma}}_h, \divg \tens{\psi}_h) = 0, \quad \forall \tens{\psi}_h \in X^s_h, \label{eq:error-eq-new-2} \\
        &\begin{aligned} \label{eq:error-eq-new-3}
            &(\pi^{\widetilde{\Sigma}, s}_h \widetilde{\matr{\sigma}} - \widetilde{\matr{\sigma}}_h, \widetilde{\matr{\tau}}_h) - (\mathcal{A}^{\frac{1}{2}} (\Pi^{\Sigma, s}_h \matr{\sigma} - \matr{\sigma}_h), \widetilde{\matr{\tau}}_h) - \iota (\divg (\pi^{X, s}_h \tens{\chi} - \tens{\chi}_h), \widetilde{\matr{\tau}}_h) \\ &\qquad = (\mathcal{A}^{\frac{1}{2}} (\matr{\sigma} - \Pi^{\Sigma, s}_h \matr{\sigma}), \widetilde{\matr{\tau}}_h) + \iota (\divg (\tens{\chi} - \pi^{X, s}_h \tens{\chi}), \widetilde{\matr{\tau}}_h), \quad \forall \widetilde{\matr{\tau}}_h \in \widetilde{\Sigma}^s_h,
        \end{aligned} \\
        &b_h (\Pi^{\Sigma, s}_h \matr{\sigma} - \matr{\sigma}_h, \vect{v}_h) = 0, \quad \forall \vect{v}_h \in U_{h, 0}. \label{eq:error-eq-new-4}
    \end{align}

    \emph{Proof of \eqref{eq:parameter-robust-1} and \eqref{eq:parameter-robust-2}.} Taking $\matr{\tau}_h = \Pi^{\Sigma, s}_h \matr{\sigma} - \matr{\sigma}_h$, $\tens{\psi}_h = \pi^{X, s}_h \tens{\chi} - \tens{\chi}_h$, $\widetilde{\matr{\tau}}_h = \pi^{\widetilde{\Sigma}, s}_h \widetilde{\matr{\sigma}} - \widetilde{\matr{\sigma}}_h$ and $\vect{v}_h = \Pi^U_h \vect{u} - \vect{u}_h$ in \eqref{eq:error-eq-new-1}--\eqref{eq:error-eq-new-4} and summing, we obtain
    \begin{align*}
        &\lVert \pi^{\widetilde{\Sigma}, s}_h \widetilde{\matr{\sigma}} - \widetilde{\matr{\sigma}}_h \rVert^2 + \lVert \pi^{X, s}_h \tens{\chi} - \tens{\chi}_h \rVert^2 \\
        &= (\mathcal{A}^{\frac{1}{2}} (\matr{\sigma} - \Pi^{\Sigma, s}_h \matr{\sigma}), \pi^{\widetilde{\Sigma}, s}_h \widetilde{\matr{\sigma}} - \widetilde{\matr{\sigma}}_h) + \iota (\divg (\tens{\chi} - \pi^{X, s}_h \tens{\chi}), \pi^{\widetilde{\Sigma}, s}_h \widetilde{\matr{\sigma}} - \widetilde{\matr{\sigma}}_h) \\
        &\lesssim \big( \lVert \matr{\sigma} - \Pi^{\Sigma, s}_h \matr{\sigma} \rVert + \iota \lVert \divg (\tens{\chi} - \pi^{X, s}_h \tens{\chi}) \rVert \big) \lVert \pi^{\widetilde{\Sigma}, s}_h \widetilde{\matr{\sigma}} - \widetilde{\matr{\sigma}}_h \rVert,
    \end{align*}
    which, with \eqref{eq:proj-err-Pi-Sigma-s-low-reg} and \eqref{eq:proj-err-pi-X-s-low-reg-div}, gives
    \begin{align}
        \lVert \pi^{\widetilde{\Sigma}, s}_h \widetilde{\matr{\sigma}} - \widetilde{\matr{\sigma}}_h \rVert + \lVert \pi^{X, s}_h \tens{\chi} - \tens{\chi}_h \rVert \lesssim h^\epsilon \iota^{1/2 - \epsilon} \lVert \vect{f} \rVert_{\epsilon} + h^{r + 1} \lVert \vect{f} \rVert_r. \label{prf:parameter-robust-1}
    \end{align}
    The argument leading to \eqref{prf:optimal-convergence-2} applies verbatim and gives
    \begin{align}
        \lVert \Pi^U_h \vect{u} - \vect{u}_h \rVert + \lVert \Pi^U_h \vect{u} - \vect{u}_h \rVert_{H^1, h} \lesssim \lVert \pi^{\widetilde{\Sigma}, s}_h \widetilde{\matr{\sigma}} - \widetilde{\matr{\sigma}}_h \rVert. \label{prf:parameter-robust-2}
    \end{align}
    Combining \eqref{prf:parameter-robust-1} and \eqref{prf:parameter-robust-2} with the projection estimates \eqref{eq:proj-err-Pi-U-low-reg}--\eqref{eq:proj-err-pi-tilde-Sigma-s-low-reg} and \eqref{eq:proj-err-pi-X-s-low-reg} gives \eqref{eq:parameter-robust-1} and \eqref{eq:parameter-robust-2}.

    \emph{Proof of \eqref{eq:parameter-robust-3}.} As in the proof of \cref{thm:optimal-convergence}, the error equation \eqref{eq:error-eq-new-4} and \cref{cor:L2-control-Sigma_h} give
    \begin{align}
        \lVert \Pi^{\Sigma, s}_h \matr{\sigma} - \matr{\sigma}_h \rVert_{L^2_*}^2 \lesssim (\mathcal{A}^{\frac{1}{2}} (\Pi^{\Sigma, s}_h \matr{\sigma} - \matr{\sigma}_h), \Pi^{\Sigma, s}_h \matr{\sigma} - \matr{\sigma}_h - c \matr{I}), \label{prf:parameter-robust-3}
    \end{align}
    where $c \defeq \frac{1}{2 \lvert \Omega \rvert} \int_{\Omega} \tr (\Pi^{\Sigma, s}_h \matr{\sigma} - \matr{\sigma}_h)$. Taking $\widetilde{\matr{\tau}}_h = \Pi^{\Sigma, s}_h \matr{\sigma} - \matr{\sigma}_h - c \matr{I}$ in \eqref{eq:error-eq-new-3} and applying the inverse inequality \eqref{eq:inverse-ineq-X_h} yields
    \begin{align*}
        \lVert \Pi^{\Sigma, s}_h \matr{\sigma} - \matr{\sigma}_h \rVert_{L^2_*}
        &\lesssim \lVert \pi^{\widetilde{\Sigma}, s}_h \widetilde{\matr{\sigma}} - \widetilde{\matr{\sigma}}_h \rVert + \frac{\iota}{h} \lVert \pi^{X, s}_h \tens{\chi} - \tens{\chi}_h \rVert \\ &\quad + \lVert \matr{\sigma} - \Pi^{\Sigma, s}_h \matr{\sigma} \rVert + \iota \lVert \divg (\tens{\chi} - \pi^{X, s}_h \tens{\chi}) \rVert,
    \end{align*}
    which, with \eqref{prf:parameter-robust-1}, \eqref{eq:proj-err-Pi-Sigma-s-low-reg} and \eqref{eq:proj-err-pi-X-s-low-reg-div}, gives
    \begin{align*}
        \lVert \matr{\sigma} - \matr{\sigma}_h \rVert_{L^2_*} \lesssim h^{\epsilon_1} \iota^{1/2 - \epsilon_1} \lVert \vect{f} \rVert_{\epsilon_1} + h^{r_1 + 1} \lVert \vect{f} \rVert_{r_1} + h^{\epsilon_2 - 1} \iota^{3/2 - \epsilon_2} \lVert \vect{f} \rVert_{\epsilon_2} + h^{r_2} \iota \lVert \vect{f} \rVert_{r_2}.
    \end{align*}
    Choosing $\epsilon_1 = \epsilon$, $\epsilon_2 = 1$ and $r_1 = r_2 = r$ gives \eqref{eq:parameter-robust-3}.
\end{proof}

\begin{corollary}[Convergence to classical elasticity] \label{cor:convergence-to-elasticity}
    Let $\vect{u}_0$ solve the classical elasticity problem \eqref{eq:classical-elasticity} and let $\matr{\sigma}_0 \defeq 2\mu \symgrad \vect{u}_0 + \lambda (\divg \vect{u}_0) \matr{I}$ be the corresponding stress. Then, under the hypotheses of \cref{thm:parameter-robust}, for $0 < \epsilon \leq 1/2$,
    \begin{align*}
        \lVert \vect{u}_h - \vect{u}_0 \rVert &\lesssim \big( \iota^{1/2} + h^\epsilon \iota^{1/2 - \epsilon} \big) \lVert \vect{f} \rVert_{\epsilon} + h^{r + 1} \lVert \vect{f} \rVert_r, & &0 \leq r \leq k, \\
        \lVert \vect{u}_h - \vect{u}_0 \rVert_{H^1, h} &\lesssim \big( \iota^{1/2} + h^\epsilon \iota^{1/2 - \epsilon} \big) \lVert \vect{f} \rVert_{\epsilon} + h^{r + 1} \lVert \vect{f} \rVert_r, & &0 \leq r \leq k - 1, \\
        \lVert \matr{\sigma}_h - \matr{\sigma}_0 \rVert_{L^2_*} &\lesssim \big( \iota^{1/2} + h^\epsilon \iota^{1/2 - \epsilon} \big) \lVert \vect{f} \rVert_1 + (h^r + \iota h^{r - 1}) \lVert \vect{f} \rVert_r, & &1 \leq r \leq k.
    \end{align*}
\end{corollary}

\begin{proof}
    Split the total error into a discretization error and a regularization error, $\vect{u}_h - \vect{u}_0 = (\vect{u}_h - \vect{u}) + (\vect{u} - \vect{u}_0)$ and $\matr{\sigma}_h - \matr{\sigma}_0 = (\matr{\sigma}_h - \matr{\sigma}) + (\matr{\sigma} - \matr{\sigma}_0)$. The discretization errors are bounded by \eqref{eq:parameter-robust-1}--\eqref{eq:parameter-robust-3}. For the regularization errors, the identity \eqref{prf:lem:proj-err-low-reg} together with the regularity hypotheses \eqref{eq:hyp-regularity-u}--\eqref{eq:hyp-regularity-u---u_0} gives the stated bounds.
\end{proof}

\begin{remark}
    The $H^2$-nonconforming finite element methods \cite{chen2023robust, liao2023taylorhood} achieve the parameter-uniform estimate
    \begin{align*}
        \lVert \vect{u} - \vect{u}_h \rVert_{1, h} + \lVert \vect{u}_h - \vect{u}_0 \rVert_{1, h} + \iota \lVert \vect{u} - \vect{u}_h \rVert_{2, h} + \iota \lVert \vect{u}_h - \vect{u}_0 \rVert_{2, h} \lesssim (\iota^{1/2} + h^{1/2}) \lVert \vect{f} \rVert.
    \end{align*}
    As discussed in the introduction, when $\iota \ll 1$ these displacement-based methods suffer from a numerical boundary layer and are therefore only half-order uniformly convergent. To recover the optimal rate in this regime, \cite{chen2025nonconforming} employs Nitsche's method, obtaining an improved bound $\iota^{1/2} \lVert \vect{f} \rVert + h^{r + 1} \lVert \vect{f} \rVert_r$ for $0 \leq r \leq k - 1$. Our method attains the optimal rate as $\iota \ll 1$ without any additional treatment. Again the price is regularity: the asymptotic term $h^\epsilon \iota^{1/2 - \epsilon} \lVert \vect{f} \rVert_\epsilon$ in \cref{thm:parameter-robust} cannot be reduced to the sharp bound $\iota^{1/2} \lVert \vect{f} \rVert$ of the $H^2$-nonconforming methods, although $\epsilon$ may be taken arbitrarily small.
\end{remark}

\section{Hybridized Discrete Scheme} \label{sec:6}

The discrete scheme of \cref{sec:3} couples the four fields globally, resulting in a large algebraic system; moreover, imposing the continuity and the symmetry constraints on the total stress simultaneously is difficult to implement. We therefore relax the inter-element continuity of the total stress and of the scaled hyper stress across the dual edges and reimpose it weakly through Lagrange multipliers. The four fields then decouple between dual elements and can be eliminated locally, leaving a global system posed only on the mesh skeleton.

Throughout this section, $\mathcal{T}_D$ denotes the set of sub-triangles of a dual element $D \in \mathcal{T}^{dl}_h$, $e_D$ the only primal edge lying inside $D$, and $\vec{\vect{n}}$ the outward unit normal vector on $\partial D$.

\subsection{Hybridization and Local Static Condensation}

Relaxing the normal continuity of $\Sigma^s_h$ across the dual edges yields the fully discontinuous space $\widetilde{\Sigma}^s_h$; relaxing that of $X^s_h$ yields the broken $H(\divg)$-conforming space $\widetilde{X}^s_h$:
\begin{align*}
    \widetilde{X}^s_h &\defeq \{ \tens{\psi} \in L^2 (\Omega; \mathbb{R}^{2 \times 2 \times 2}) : \tens{\psi}_{ij\bullet}|_D \in RT_k (\mathcal{T}_D),\ \forall D \in \mathcal{T}^{dl}_h;\ \tens{\psi}_{12\bullet} = \tens{\psi}_{21\bullet} \}.
\end{align*}
The associated multipliers $\widehat{\vect{u}}_h$ and $\widehat{\matr{\sigma}}_h$, which approximate the traces of $\vect{u}$ and of $\iota \, \widetilde{\matr{\sigma}}$ on the dual edges, live in $\widehat{U}_h$ and $\widehat{\Sigma}^s_h$:
\begin{align*}
    \widehat{U}_h &\defeq \{ \widehat{\vect{v}} \in L^2 (\mathcal{F}^{dl}_h; \mathbb{R}^2) : \widehat{\vect{v}}|_e \in \mathcal{P}_k (e; \mathbb{R}^2),\ \forall e \in \mathcal{F}^{dl}_h \}, \\
    \widehat{\Sigma}^s_h &\defeq \{ \widehat{\matr{\tau}} \in L^2 (\mathcal{F}^{dl}_h; \mathbb{R}^{2 \times 2}) : \widehat{\tau}_{ij}|_e \in \mathcal{P}_k (e),\ \forall e \in \mathcal{F}^{dl}_h;\ \widehat{\tau}_{12} = \widehat{\tau}_{21} \}.
\end{align*}
To accommodate the fully discontinuous space $\widetilde{\Sigma}^s_h$, we extend $b_h$ to
\begin{align}
    b_h (\matr{\tau}_h, \vect{v}_h) &\defeq (\matr{\tau}_h, \grad \vect{v}_h)_{\mathcal{T}^{dl}_h} - \sum_{e \in \mathcal{F}^{dl}_h} \int_e \jump{(\matr{\tau}_h \cdot \vect{n}) \cdot \vect{v}_h} \dif s, \label{eq:b_h-hybrid}
\end{align}
which agrees with the original definition \eqref{eq:b_h} for $\matr{\tau}_h \in \Sigma^s_h$.

\paragraph{Hybridized discrete scheme.} Find $(\vect{u}_h, \matr{\sigma}_h, \widetilde{\matr{\sigma}}_h, \tens{\chi}_h, \widehat{\vect{u}}_h, \widehat{\matr{\sigma}}_h) \in U_{h, 0} \times \widetilde{\Sigma}^s_h \times \widetilde{\Sigma}^s_h \times \widetilde{X}^s_h \times \widehat{U}_h \times \widehat{\Sigma}^s_h$ such that
\begin{align}
    (\mathcal{A}^{\frac{1}{2}} \widetilde{\matr{\sigma}}_h, \matr{\tau}_h) - b^*_h (\vect{u}_h, \matr{\tau}_h) - ( \widehat{\vect{u}}_h, \jump{\matr{\tau}_h \cdot \vect{n}} )_{\mathcal{F}^{dl}_h} &= 0, & &\forall \matr{\tau}_h \in \widetilde{\Sigma}^s_h, \label{eq:hybrid-scheme-1} \\
    (\tens{\chi}_h, \tens{\psi}_h) + \iota (\widetilde{\matr{\sigma}}_h, \divg \tens{\psi}_h)_{\mathcal{T}^{dl}_h} - ( \widehat{\matr{\sigma}}_h, \jump{\tens{\psi}_h \cdot \vect{n}} )_{\mathcal{F}^{dl}_h} &= 0, & &\forall \tens{\psi}_h \in \widetilde{X}^s_h, \label{eq:hybrid-scheme-2} \\
    (\widetilde{\matr{\sigma}}_h, \widetilde{\matr{\tau}}_h) - (\mathcal{A}^{\frac{1}{2}} \matr{\sigma}_h, \widetilde{\matr{\tau}}_h) - \iota (\divg \tens{\chi}_h, \widetilde{\matr{\tau}}_h)_{\mathcal{T}^{dl}_h} &= 0, & &\forall \widetilde{\matr{\tau}}_h \in \widetilde{\Sigma}^s_h, \label{eq:hybrid-scheme-3} \\
    b_h (\matr{\sigma}_h, \vect{v}_h) &= (\vect{f}, \vect{v}_h), & &\forall \vect{v}_h \in U_{h, 0}, \label{eq:hybrid-scheme-4} \\
    ( \jump{\matr{\sigma}_h \cdot \vect{n}}, \widehat{\vect{v}}_h )_{\mathcal{F}^{dl}_h} &= 0, & &\forall \widehat{\vect{v}}_h \in \widehat{U}_h, \label{eq:hybrid-scheme-5} \\
    ( \jump{\tens{\chi}_h \cdot \vect{n}}, \widehat{\matr{\tau}}_h )_{\mathcal{F}^{dl}_h} &= 0, & &\forall \widehat{\matr{\tau}}_h \in \widehat{\Sigma}^s_h. \label{eq:hybrid-scheme-6}
\end{align}

\begin{proposition}[Equivalence] \label{prop:equivalence}
    Let $(\vect{u}_h, \matr{\sigma}_h, \widetilde{\matr{\sigma}}_h, \tens{\chi}_h, \widehat{\vect{u}}_h, \widehat{\matr{\sigma}}_h)$ solve the hybridized scheme \eqref{eq:hybrid-scheme-1}--\eqref{eq:hybrid-scheme-6}. Then $(\vect{u}_h, \matr{\sigma}_h, \widetilde{\matr{\sigma}}_h, \tens{\chi}_h) \in U_{h, 0} \times \Sigma^s_h \times \widetilde{\Sigma}^s_h \times X^s_h$ solves the original scheme \eqref{eq:discrete-scheme-1}--\eqref{eq:discrete-scheme-4}.
\end{proposition}

\begin{proof}
    Equations \eqref{eq:hybrid-scheme-5} and \eqref{eq:hybrid-scheme-6} enforce $\jump{\matr{\sigma}_h \cdot \vect{n}}_{|_e} = \vect{0}$ and $\jump{\tens{\chi}_h \cdot \vect{n}}_{|_e} = \matr{0}$ for every $e \in \mathcal{F}^{dl}_h$, so that $\matr{\sigma}_h \in \Sigma^s_h$ and $\tens{\chi}_h \in X^s_h$. Restricting the test functions in \eqref{eq:hybrid-scheme-1} and \eqref{eq:hybrid-scheme-2} to $\matr{\tau}_h \in \Sigma^s_h$ and $\tens{\psi}_h \in X^s_h$ makes the multiplier terms vanish, and \eqref{eq:hybrid-scheme-1}--\eqref{eq:hybrid-scheme-4} reduce to \eqref{eq:discrete-scheme-1}--\eqref{eq:discrete-scheme-4}.
\end{proof}

The spaces $U_h$, $\widetilde{\Sigma}^s_h$ and $\widetilde{X}^s_h$ carry no continuity constraint across the dual edges, and there is no coupling term on the dual edges among $\vect{u}_h$, $\widetilde{\matr{\sigma}}_h$, $\matr{\sigma}_h$, and $\tens{\chi}_h$ in the hybridized scheme. These four fields can therefore be eliminated by solving independent problems on each dual element. Writing $U_h (D)$, $\widetilde{\Sigma}^s_h (D)$ and $\widetilde{X}^s_h (D)$ for the restrictions of the corresponding spaces to $D$, and recalling the definitions \eqref{eq:b*_h} and \eqref{eq:b_h-hybrid}, we localize \eqref{eq:hybrid-scheme-1}--\eqref{eq:hybrid-scheme-4} as follows.

\paragraph{Local problem.}
Given $(\widehat{\vect{u}}_h, \widehat{\matr{\sigma}}_h)$, on each $D \in \mathcal{T}^{dl}_h$ find $(\vect{u}_h, \matr{\sigma}_h, \widetilde{\matr{\sigma}}_h, \tens{\chi}_h) \in U_{h, 0} (D) \times \widetilde{\Sigma}^s_h (D) \times \widetilde{\Sigma}^s_h (D) \times \widetilde{X}^s_h (D)$ such that
\begin{align*}
    (\mathcal{A}^{\frac{1}{2}} \widetilde{\matr{\sigma}}_h, \matr{\tau}_h)_D + (\vect{u}_h, \divg \matr{\tau}_h)_{\mathcal{T}_D} - (\vect{u}_h, \jump{\matr{\tau}_h \cdot \vect{n}})_{e_D} &= (\widehat{\vect{u}}_h, \matr{\tau}_h \cdot \vec{\vect{n}})_{\partial D}, \\
    (\tens{\chi}_h, \tens{\psi}_h)_D + \iota (\widetilde{\matr{\sigma}}_h, \divg \tens{\psi}_h)_D &= (\widehat{\matr{\sigma}}_h, \tens{\psi}_h \cdot \vec{\vect{n}})_{\partial D}, \\
    (\widetilde{\matr{\sigma}}_h, \widetilde{\matr{\tau}}_h)_D - (\mathcal{A}^{\frac{1}{2}} \matr{\sigma}_h, \widetilde{\matr{\tau}}_h)_D - \iota (\divg \tens{\chi}_h, \widetilde{\matr{\tau}}_h)_D &= 0, \\
    (\matr{\sigma}_h, \grad \vect{v}_h)_{\mathcal{T}_D} - (\matr{\sigma}_h \cdot \vec{\vect{n}}, \vect{v}_h)_{\partial D} &= (\vect{f}, \vect{v}_h)_D,
\end{align*}
for all $(\matr{\tau}_h, \tens{\psi}_h, \widetilde{\matr{\tau}}_h, \vect{v}_h) \in \widetilde{\Sigma}^s_h (D) \times \widetilde{X}^s_h (D) \times \widetilde{\Sigma}^s_h (D) \times U_{h, 0} (D)$.

Writing $\matr{\sigma}^{(\widehat{\vect{u}}_h, \widehat{\matr{\sigma}}_h)}_h$ and $\tens{\chi}^{(\widehat{\vect{u}}_h, \widehat{\matr{\sigma}}_h)}_h$ for the fields obtained by collecting the local solutions over all dual elements and substituting them into \eqref{eq:hybrid-scheme-5}--\eqref{eq:hybrid-scheme-6} leaves a global system in the multipliers alone.

\paragraph{Global problem.} Find $(\widehat{\vect{u}}_h, \widehat{\matr{\sigma}}_h) \in \widehat{U}_h \times \widehat{\Sigma}^s_h$ such that
\begin{align*}
    (\jump{\matr{\sigma}^{(\widehat{\vect{u}}_h, \widehat{\matr{\sigma}}_h)}_h \cdot \vect{n}}, \widehat{\vect{v}}_h)_{\mathcal{F}^{dl}_h} &= 0, & &\forall \widehat{\vect{v}}_h \in \widehat{U}_h, \\
    (\jump{\tens{\chi}^{(\widehat{\vect{u}}_h, \widehat{\matr{\sigma}}_h)}_h \cdot \vect{n}}, \widehat{\matr{\tau}}_h)_{\mathcal{F}^{dl}_h} &= 0, & &\forall \widehat{\matr{\tau}}_h \in \widehat{\Sigma}^s_h.
\end{align*}

\subsection{Stability}

By the equivalence in \cref{prop:equivalence}, the stability of $\vect{u}_h$, $\matr{\sigma}_h$, $\widetilde{\matr{\sigma}}_h$ and $\tens{\chi}_h$ is inherited from \cref{thm:stability}, and it remains to control the two multipliers. To this end we use the norms
\begin{align*}
    \lVert (\vect{v}_h, \widehat{\vect{v}}_h) \rVert^2_{H^1, h} &\defeq \sum_{K \in \mathcal{T}_h} \lVert \grad \vect{v}_h \rVert^2_K + \sum_{e \in \mathcal{F}^{dl}_h} h_e^{-1} \lVert \vect{v}_h - \widehat{\vect{v}}_h \rVert^2_e, &
    \lVert \widehat{\matr{\tau}}_h \rVert^2_{L^2, h} &\defeq \sum_{e \in \mathcal{F}^{dl}_h} h_e \lVert \widehat{\matr{\tau}}_h \rVert^2_e,
\end{align*}
for $(\vect{v}_h, \widehat{\vect{v}}_h) \in U_{h, 0} \times \widehat{U}_h$ and $\widehat{\matr{\tau}}_h \in \widehat{\Sigma}^s_h$, and establish an inf-sup condition for each multiplier.

\begin{lemma}[Inf-sup condition for $\widehat{\vect{u}}_h$] \label{lem:inf-sup-b*_h-hybrid}
    There exists a constant $\beta > 0$ independent of the mesh size $h$ such that
    \begin{align}
        \sup_{\widetilde{\matr{\tau}}_h \in \widetilde{\Sigma}^s_h \setminus \{ \matr{0} \}} \frac{b^*_h (\vect{v}_h, \widetilde{\matr{\tau}}_h) + (\widehat{\vect{v}}_h, \jump{\widetilde{\matr{\tau}}_h \cdot \vect{n}})_{\mathcal{F}^{dl}_h}}{\lVert \widetilde{\matr{\tau}}_h \rVert} &\geq \beta \lVert (\vect{v}_h, \widehat{\vect{v}}_h) \rVert_{H^1, h} \label{eq:inf-sup-b*_h-hybrid}
    \end{align}
    for all $(\vect{v}_h, \widehat{\vect{v}}_h) \in U_{h, 0} \times \widehat{U}_h$.
\end{lemma}

\begin{proof}
    It suffices to construct, for each $(\vect{v}_h, \widehat{\vect{v}}_h) \in U_{h, 0} \times \widehat{U}_h$, some $\widetilde{\matr{\tau}}^s_h \in \widetilde{\Sigma}^s_h$ with
    \begin{align}
        b^*_h (\vect{v}_h, \widetilde{\matr{\tau}}^s_h) + (\widehat{\vect{v}}_h, \jump{\widetilde{\matr{\tau}}^s_h \cdot \vect{n}})_{\mathcal{F}^{dl}_h} \gtrsim \lVert (\vect{v}_h, \widehat{\vect{v}}_h) \rVert^2_{H^1, h}, \qquad
        \lVert \widetilde{\matr{\tau}}^s_h \rVert \lesssim \lVert (\vect{v}_h, \widehat{\vect{v}}_h) \rVert_{H^1, h}. \label{prf:inf-sup-b*_h-hybrid-0}
    \end{align}
    We first construct such a field in the discontinuous space without symmetry constraint,
    \begin{align*}
        \widetilde{\Sigma}_h &\defeq \{ \widetilde{\matr{\tau}} \in L^2 (\Omega; \mathbb{R}^{2 \times 2}) : \widetilde{\matr{\tau}}|_K \in \mathcal{P}_k (K; \mathbb{R}^{2 \times 2}),\ \forall K \in \mathcal{T}_h \},
    \end{align*}
    and then symmetrize it.

    \emph{Step 1: construction in $\widetilde{\Sigma}_h$.} Integrating by parts in the definition \eqref{eq:b*_h} of $b^*_h$ gives, for every $\widetilde{\matr{\tau}}_h \in \widetilde{\Sigma}_h$,
    \begin{align}
        b^*_h (\vect{v}_h, \widetilde{\matr{\tau}}_h) + (\widehat{\vect{v}}_h, \jump{\widetilde{\matr{\tau}}_h \cdot \vect{n}})_{\mathcal{F}^{dl}_h} = (\widetilde{\matr{\tau}}_h, \grad \vect{v}_h)_{\mathcal{T}_h} - \sum_{K \in \mathcal{T}_h} (\vect{v}_h - \widehat{\vect{v}}_h, \widetilde{\matr{\tau}}_h \cdot \vec{\vect{n}})_{\partial K \cap \mathcal{F}^{dl}_h}. \label{prf:inf-sup-b*_h-hybrid-1}
    \end{align}
    By the degrees of freedom of $\Sigma_h$, on each $K \in \mathcal{T}_h$ there exists $\widetilde{\matr{\tau}}_K \in \mathcal{P}_k (K; \mathbb{R}^{2 \times 2})$ such that
    \begin{align*}
        (\widetilde{\matr{\tau}}_K, \matr{\phi})_K &= (\grad \vect{v}_h, \matr{\phi})_K, & &\forall \matr{\phi} \in \mathcal{P}_{k - 1} (K; \mathbb{R}^{2 \times 2}), \\
        (\widetilde{\matr{\tau}}_K \cdot \vec{\vect{n}}, \vect{\phi})_e &= - h^{-1}_e (\vect{v}_h - \widehat{\vect{v}}_h, \vect{\phi})_e, & &\forall \vect{\phi} \in \mathcal{P}_k (e; \mathbb{R}^2),\ \forall e \in \partial K \cap \mathcal{F}^{dl}_h.
    \end{align*}
    Writing $S_K \defeq \lVert \grad \vect{v}_h \rVert^2_K + \sum_{e \in \partial K \cap \mathcal{F}^{dl}_h} h^{-1}_e \lVert \vect{v}_h - \widehat{\vect{v}}_h \rVert^2_e$, these two conditions give
    \begin{align*}
        (\widetilde{\matr{\tau}}_K, \grad \vect{v}_h)_K - (\vect{v}_h - \widehat{\vect{v}}_h, \widetilde{\matr{\tau}}_K \cdot \vec{\vect{n}})_{\partial K \cap \mathcal{F}^{dl}_h} = S_K,
    \end{align*}
    while a standard scaling argument (see \cite{chung2024stabilizationfree} for reference) gives $\lVert \widetilde{\matr{\tau}}_K \rVert^2_K \lesssim S_K$. Assembling over all $K \in \mathcal{T}_h$ and recalling \eqref{prf:inf-sup-b*_h-hybrid-1}, we obtain $\widetilde{\matr{\tau}}_h \in \widetilde{\Sigma}_h$ with
    \begin{align}
        b^*_h (\vect{v}_h, \widetilde{\matr{\tau}}_h) + (\widehat{\vect{v}}_h, \jump{\widetilde{\matr{\tau}}_h \cdot \vect{n}})_{\mathcal{F}^{dl}_h} = \lVert (\vect{v}_h, \widehat{\vect{v}}_h) \rVert^2_{H^1, h}, \qquad
        \lVert \widetilde{\matr{\tau}}_h \rVert \lesssim \lVert (\vect{v}_h, \widehat{\vect{v}}_h) \rVert_{H^1, h}. \label{prf:inf-sup-b*_h-hybrid-2}
    \end{align}

    \emph{Step 2: symmetrization.} As in \cref{lem:symmetrization}, there is an operator $\mathcal{S} : \widetilde{\Sigma}_h \to \widetilde{\Sigma}^s_h$ of the form $\mathcal{S} \widetilde{\matr{\tau}}_h = \widetilde{\matr{\tau}}_h + \curl \vect{w}_h$ with $\vect{w}_h \in H^1 (\Omega; \mathbb{R}^2)$, satisfying $b^*_h (\vect{v}_h, \mathcal{S} \widetilde{\matr{\tau}}_h) = b^*_h (\vect{v}_h, \widetilde{\matr{\tau}}_h)$ for all $\vect{v}_h \in U_{h, 0}$ and $\lVert \mathcal{S} \widetilde{\matr{\tau}}_h \rVert \lesssim \lVert \widetilde{\matr{\tau}}_h \rVert$. The conformity $\curl \vect{w}_h \in H (\divg; \Omega)$ gives $\jump{\curl \vect{w}_h \cdot \vect{n}}_{|_e} = \vect{0}$ for every $e \in \mathcal{F}^{dl}_h$, so the multiplier term is unchanged as well. Hence $\widetilde{\matr{\tau}}^s_h \defeq \mathcal{S} \widetilde{\matr{\tau}}_h$ inherits \eqref{prf:inf-sup-b*_h-hybrid-2} and satisfies \eqref{prf:inf-sup-b*_h-hybrid-0}.
\end{proof}

The second inf-sup condition is a direct application of the hybridization theory of the Raviart--Thomas element \cite{arnold1985mixed, boffi2013mixed}. Indeed, $\widetilde{X}^s_h$ is built from the broken Raviart--Thomas space
\begin{align*}
    \widetilde{RT}_k (\mathcal{T}_h) \defeq \{ \vect{w}|_D \in H (\divg; D),\ \forall D \in \mathcal{T}^{dl}_h;\ \vect{w}|_K \in \mathcal{P}_k (K; \mathbb{R}^2) \oplus \vect{x} \widetilde{\mathcal{P}}_k (K),\ \forall K \in \mathcal{T}_h \}
\end{align*}
by replicating it along the first two indices and imposing the symmetry structure, while $\widetilde{\Sigma}^s_h \times \widehat{\Sigma}^s_h$ is built analogously from $\mathcal{P}_k$ spaces. Applying that theory componentwise to each index pair $(i, j)$ and assembling gives a constant $\beta > 0$, independent of $h$, such that
\begin{align}
    \sup_{\tens{\psi}_h \in \widetilde{X}^s_h \setminus \{ \tens{0} \}} \frac{\iota (\widetilde{\matr{\tau}}_h, \divg \tens{\psi}_h)_{\mathcal{T}^{dl}_h} - (\widehat{\matr{\tau}}_h, \jump{\tens{\psi}_h \cdot \vect{n}})_{\mathcal{F}^{dl}_h}}{\lVert \tens{\psi}_h \rVert} \geq \beta \left( \iota \lVert \widetilde{\matr{\tau}}_h \rVert + \lVert \widehat{\matr{\tau}}_h \rVert_{L^2, h} \right) \label{eq:inf-sup-B^*_h-hybrid}
\end{align}
for all $(\widetilde{\matr{\tau}}_h, \widehat{\matr{\tau}}_h) \in \widetilde{\Sigma}^s_h \times \widehat{\Sigma}^s_h$.

\begin{theorem}[Stability of the hybridized scheme] \label{thm:stability-hybrid}
    Let $(\vect{u}_h, \matr{\sigma}_h, \widetilde{\matr{\sigma}}_h, \tens{\chi}_h, \widehat{\vect{u}}_h, \widehat{\matr{\sigma}}_h)$ solve the hybridized scheme \eqref{eq:hybrid-scheme-1}--\eqref{eq:hybrid-scheme-6}. Then, in addition to \eqref{eq:stability-1} and \eqref{eq:stability-2}, the multipliers satisfy
    \begin{align}
        \lVert (\vect{u}_h, \widehat{\vect{u}}_h) \rVert_{H^1, h} + \lVert \widehat{\matr{\sigma}}_h \rVert_{L^2, h} &\lesssim \lVert \vect{f} \rVert. \label{eq:stability-hybrid-1}
    \end{align}
\end{theorem}

\begin{proof}
    By \cref{prop:equivalence}, the four fields $\vect{u}_h$, $\matr{\sigma}_h$, $\widetilde{\matr{\sigma}}_h$ and $\tens{\chi}_h$ solve the original scheme, so that \eqref{eq:stability-1} and \eqref{eq:stability-2} hold. The inf-sup condition \eqref{eq:inf-sup-b*_h-hybrid} and the equation \eqref{eq:hybrid-scheme-1} then give
    \begin{align*}
        \lVert (\vect{u}_h, \widehat{\vect{u}}_h) \rVert_{H^1, h} \lesssim \sup_{\matr{\tau}_h \in \widetilde{\Sigma}^s_h \setminus \{ \matr{0} \}} \frac{(\mathcal{A}^{\frac{1}{2}} \widetilde{\matr{\sigma}}_h, \matr{\tau}_h)}{\lVert \matr{\tau}_h \rVert} \lesssim \lVert \widetilde{\matr{\sigma}}_h \rVert,
    \end{align*}
    while \eqref{eq:inf-sup-B^*_h-hybrid} and \eqref{eq:hybrid-scheme-2} give
    \begin{align*}
        \lVert \widehat{\matr{\sigma}}_h \rVert_{L^2, h} \lesssim \sup_{\tens{\psi}_h \in \widetilde{X}^s_h \setminus \{ \tens{0} \}} \frac{- (\tens{\chi}_h, \tens{\psi}_h)}{\lVert \tens{\psi}_h \rVert} \lesssim \lVert \tens{\chi}_h \rVert.
    \end{align*}
    Both right-hand sides are bounded by $\lVert \vect{f} \rVert$ through \eqref{eq:stability-1}.
\end{proof}

\begin{remark} \label{rmk:incompressible-hybrid}
    In the incompressible limit $\lambda \to \infty$, \cref{rmk:incompressible-limit} requires the normalization $\int_{\Omega} \tr \matr{\sigma}_h = 0$, which is global and therefore obstructs the local elimination of $\matr{\sigma}_h$. The remedy is to split $\matr{\sigma}_h = \matr{\sigma}^0_h + \overline{\sigma}_h \matr{I}$, where $\matr{\sigma}^0_h$ has elementwise zero-mean trace and $\overline{\sigma}_h$ is piecewise constant with zero mean over $\Omega$. The first part retains the local structure of \eqref{eq:hybrid-scheme-1}--\eqref{eq:hybrid-scheme-6}, while $\overline{\sigma}_h$ joins $(\widehat{\vect{u}}_h, \widehat{\matr{\sigma}}_h)$ as a third globally coupled unknown. The resulting incompressible hybridized scheme and its local and global problems are given in \Cref{sm:incompressible}.
\end{remark}

\section{Numerical Experiments} \label{sec:7}

We report two numerical experiments verifying the convergence rates and the parameter robustness established above; for ease of comparison we adopt the examples of \cite{chen2023robust, chen2025nonconforming}. The first has a smooth solution, the second a solution with strong boundary layers. In both we take $\Omega = (0,1)^2$, polynomial degree $k = 1$ and shear modulus $\mu = 1$, and obtain the primal mesh $\mathcal{T}^{pr}_h$ by partitioning $\Omega$ into a uniform $N \times N$ array of squares, each split along its diagonal into two triangles, so that $h = 1/N$. We sweep over $N \in \{ 8, 16, 32, 64, 128 \}$ and $\lambda \in \{ 1, 10^4, 10^8 \}$; in the nearly incompressible cases $\lambda = 10^4$ and $\lambda = 10^8$ we impose the additional constraint $\int_{\Omega} \tr \matr{\sigma}_h = 0$ of \cref{rmk:incompressible-limit,rmk:incompressible-hybrid}. Each reported rate is computed from the errors on the two finest meshes.

\subsection{Convergence and Parameter Robustness} \label{sec:example-1}

The source term $\vect{f}$ is chosen so that the exact solution of \eqref{eq:SGE} is
\begin{align*}
    \vect{u} = \begin{pmatrix}
        3 (e^{\cos (2\pi x)} - e)^2 \sin (2\pi y) \sin (\pi y) \\
        8 (e^{2 \cos (2\pi x)} - e^{1 + \cos (2\pi x)}) \sin (2\pi x) \sin^3 (\pi y)
    \end{pmatrix}.
\end{align*}
This $\vect{u}$ is smooth and divergence free, so that $\vect{f}$ is independent of $\lambda$ and the solution has no boundary layer.

Here $\iota \in \{ 1, 10^{-4}, 10^{-8} \}$. \Cref{tab:1,tab:2} report the displacement error in the $L^2$ and discrete $H^1$ norms; both confirm optimal convergence and robustness in $\lambda$ and $\iota$. \Cref{tab:3} reports the total stress in the $L^2$ norm, where the convergence is of second order for $\iota = 10^{-4}$ and $\iota = 10^{-8}$ and drops to first order for $\iota = 1$. All results agree with \cref{thm:optimal-convergence}.

\begin{table}[htbp]
    \centering
    \caption{Relative error $\frac{\lVert \vect{u}-\vect{u}_h \rVert}{\lVert \vect{u} \rVert}$ and convergence rate in Example~\ref{sec:example-1}.}
    \label{tab:1}
    \footnotesize
    \begin{tabular}{ccccccc c}
        \toprule
        & & \multicolumn{5}{c}{$h$} & \\
        \cmidrule(lr){3-7}
        $\lambda$ & $\iota$ & $1/8$ & $1/16$ & $1/32$ & $1/64$ & $1/128$ & Rate \\
        \midrule
        \multirow{3}{*}{$1e+00$} & $1e+00$ & 3.617e-02 & 7.749e-03 & 1.938e-03 & 4.851e-04 & 1.213e-04 & 2.00 \\
        & $1e-04$ & 3.952e-02 & 9.976e-03 & 2.243e-03 & 5.145e-04 & 1.235e-04 & 2.06 \\
        & $1e-08$ & 3.952e-02 & 9.977e-03 & 2.244e-03 & 5.147e-04 & 1.236e-04 & 2.06 \\
        \midrule
        \multirow{3}{*}{$1e+04$} & $1e+00$ & 3.611e-02 & 7.748e-03 & 1.938e-03 & 4.851e-04 & 1.213e-04 & 2.00 \\
        & $1e-04$ & 3.701e-02 & 9.238e-03 & 2.164e-03 & 5.086e-04 & 1.231e-04 & 2.05 \\
        & $1e-08$ & 3.701e-02 & 9.238e-03 & 2.165e-03 & 5.087e-04 & 1.232e-04 & 2.05 \\
        \midrule
        \multirow{3}{*}{$1e+08$} & $1e+00$ & 3.611e-02 & 7.748e-03 & 1.938e-03 & 4.851e-04 & 1.213e-04 & 2.00 \\
        & $1e-04$ & 3.701e-02 & 9.238e-03 & 2.164e-03 & 5.086e-04 & 1.231e-04 & 2.05 \\
        & $1e-08$ & 3.701e-02 & 9.238e-03 & 2.165e-03 & 5.087e-04 & 1.232e-04 & 2.05 \\
        \bottomrule
    \end{tabular}
\end{table}

\begin{table}[htbp]
    \centering
    \caption{Relative error $\frac{\lVert \vect{u}-\vect{u}_h \rVert_{H^1, h}}{\lVert \vect{u} \rVert_{H^1, h}}$ and convergence rate in Example~\ref{sec:example-1}.}
    \label{tab:2}
    \footnotesize
    \begin{tabular}{ccccccc c}
        \toprule
        & & \multicolumn{5}{c}{$h$} & \\
        \cmidrule(lr){3-7}
        $\lambda$ & $\iota$ & $1/8$ & $1/16$ & $1/32$ & $1/64$ & $1/128$ & Rate \\
        \midrule
        \multirow{3}{*}{$1e+00$} & $1e+00$ & 5.641e-01 & 2.691e-01 & 1.358e-01 & 6.803e-02 & 3.403e-02 & 1.00 \\
        & $1e-04$ & 6.121e-01 & 2.993e-01 & 1.428e-01 & 6.928e-02 & 3.421e-02 & 1.02 \\
        & $1e-08$ & 6.121e-01 & 2.993e-01 & 1.428e-01 & 6.929e-02 & 3.422e-02 & 1.02 \\
        \midrule
        \multirow{3}{*}{$1e+04$} & $1e+00$ & 5.631e-01 & 2.690e-01 & 1.358e-01 & 6.803e-02 & 3.403e-02 & 1.00 \\
        & $1e-04$ & 5.883e-01 & 2.886e-01 & 1.408e-01 & 6.900e-02 & 3.418e-02 & 1.01 \\
        & $1e-08$ & 5.883e-01 & 2.887e-01 & 1.409e-01 & 6.901e-02 & 3.418e-02 & 1.01 \\
        \midrule
        \multirow{3}{*}{$1e+08$} & $1e+00$ & 5.631e-01 & 2.690e-01 & 1.358e-01 & 6.803e-02 & 3.403e-02 & 1.00 \\
        & $1e-04$ & 5.883e-01 & 2.886e-01 & 1.408e-01 & 6.900e-02 & 3.418e-02 & 1.01 \\
        & $1e-08$ & 5.883e-01 & 2.887e-01 & 1.409e-01 & 6.901e-02 & 3.418e-02 & 1.01 \\
        \bottomrule
    \end{tabular}
\end{table}

\begin{table}[htbp]
    \centering
    \caption{Relative error $\frac{\lVert \matr{\sigma} - \matr{\sigma}_h \rVert}{\lVert \matr{\sigma} \rVert}$ and convergence rate in Example~\ref{sec:example-1}.}
    \label{tab:3}
    \footnotesize
    \begin{tabular}{ccccccc c}
        \toprule
        & & \multicolumn{5}{c}{$h$} & \\
        \cmidrule(lr){3-7}
        $\lambda$ & $\iota$ & $1/8$ & $1/16$ & $1/32$ & $1/64$ & $1/128$ & Rate \\
        \midrule
        \multirow{3}{*}{$1e+00$} & $1e+00$ & 4.379e-01 & 2.136e-01 & 1.075e-01 & 5.411e-02 & 2.741e-02 & 0.98 \\
        & $1e-04$ & 1.010e-01 & 4.260e-02 & 1.230e-02 & 3.317e-03 & 8.564e-04 & 1.95 \\
        & $1e-08$ & 1.010e-01 & 4.260e-02 & 1.230e-02 & 3.317e-03 & 8.565e-04 & 1.95 \\
        \midrule
        \multirow{3}{*}{$1e+04$} & $1e+00$ & 4.576e-01 & 2.256e-01 & 1.153e-01 & 5.836e-02 & 3.060e-02 & 0.93 \\
        & $1e-04$ & 1.027e-01 & 4.334e-02 & 1.252e-02 & 3.356e-03 & 8.611e-04 & 1.96 \\
        & $1e-08$ & 1.027e-01 & 4.334e-02 & 1.252e-02 & 3.356e-03 & 8.612e-04 & 1.96 \\
        \midrule
        \multirow{3}{*}{$1e+08$} & $1e+00$ & 4.576e-01 & 2.256e-01 & 1.153e-01 & 5.836e-02 & 3.072e-02 & 0.93 \\
        & $1e-04$ & 1.027e-01 & 4.334e-02 & 1.252e-02 & 3.356e-03 & 8.611e-04 & 1.96 \\
        & $1e-08$ & 1.027e-01 & 4.334e-02 & 1.252e-02 & 3.356e-03 & 8.612e-04 & 1.96 \\
        \bottomrule
    \end{tabular}
\end{table}

\subsection{Performance under Boundary Layers} \label{sec:example-2}

We take $\vect{f}$ to be the source term of the classical elasticity problem \eqref{eq:classical-elasticity} with the divergence-free exact solution
\begin{align*}
    \vect{u}_0 = \begin{pmatrix}
        (e^{\cos (2\pi x)} - e) \sin (2\pi y) \, e^{\cos (2\pi y)} \\
        - (e^{\cos (2\pi y)} - e) \sin (2\pi x) \, e^{\cos (2\pi x)}
    \end{pmatrix},
\end{align*}
so that $\vect{f}$ is independent of both $\lambda$ and $\iota$, and use it as the source term of the strain gradient elasticity problem \eqref{eq:SGE}. Since $\partial_{\vect{n}} \vect{u}_0 |_{\partial \Omega} \neq \vect{0}$, the solution $\vect{u}$ of \eqref{eq:SGE} develops a strong boundary layer for small $\iota$.

The exact solution $\vect{u}$ is unknown here, so we measure the error against $\vect{u}_0$ instead: as in the proof of \cref{cor:convergence-to-elasticity}, the convergence of $(\vect{u}_h, \matr{\sigma}_h)$ to $(\vect{u}_0, \matr{\sigma}_0)$ together with the regularity hypotheses \eqref{eq:hyp-regularity-u}--\eqref{eq:hyp-regularity-u---u_0} implies convergence to $(\vect{u}, \matr{\sigma})$.

Here $\iota \in \{ 10^{-6}, 10^{-8} \}$. \Cref{tab:4,tab:5,tab:6} report the errors and rates of the displacement and the total stress; the method retains optimal convergence and parameter robustness in the presence of boundary layers, in agreement with \cref{thm:parameter-robust,cor:convergence-to-elasticity}.

\begin{table}[htbp]
    \centering
    \caption{Relative error $\frac{\lVert \vect{u}_h-\vect{u}_0 \rVert}{\lVert \vect{u}_0 \rVert}$ and convergence rate in Example~\ref{sec:example-2}.}
    \label{tab:4}
    \footnotesize
    \begin{tabular}{ccccccc c}
        \toprule
        & & \multicolumn{5}{c}{$h$} & \\
        \cmidrule(lr){3-7}
        $\lambda$ & $\iota$ & $1/8$ & $1/16$ & $1/32$ & $1/64$ & $1/128$ & Rate \\
        \midrule
        \multirow{2}{*}{$1e+00$} & $1e-06$ & 3.908e-02 & 9.188e-03 & 2.131e-03 & 5.045e-04 & 1.232e-04 & 2.03 \\
        & $1e-08$ & 3.908e-02 & 9.188e-03 & 2.131e-03 & 5.045e-04 & 1.232e-04 & 2.03 \\
        \midrule
        \multirow{2}{*}{$1e+04$} & $1e-06$ & 3.594e-02 & 8.719e-03 & 2.084e-03 & 5.011e-04 & 1.229e-04 & 2.03 \\
        & $1e-08$ & 3.594e-02 & 8.719e-03 & 2.084e-03 & 5.011e-04 & 1.229e-04 & 2.03 \\
        \midrule
        \multirow{2}{*}{$1e+08$} & $1e-06$ & 3.594e-02 & 8.718e-03 & 2.084e-03 & 5.011e-04 & 1.229e-04 & 2.03 \\
        & $1e-08$ & 3.594e-02 & 8.718e-03 & 2.084e-03 & 5.011e-04 & 1.229e-04 & 2.03 \\
        \bottomrule
    \end{tabular}
\end{table}

\begin{table}[htbp]
    \centering
    \caption{Relative error $\frac{\lVert \vect{u}_h-\vect{u}_0 \rVert_{H^1, h}}{\lVert \vect{u}_0 \rVert_{H^1, h}}$ and convergence rate in Example~\ref{sec:example-2}.}
    \label{tab:5}
    \footnotesize
    \begin{tabular}{ccccccc c}
        \toprule
        & & \multicolumn{5}{c}{$h$} & \\
        \cmidrule(lr){3-7}
        $\lambda$ & $\iota$ & $1/8$ & $1/16$ & $1/32$ & $1/64$ & $1/128$ & Rate \\
        \midrule
        \multirow{2}{*}{$1e+00$} & $1e-06$ & 5.665e-01 & 2.749e-01 & 1.334e-01 & 6.553e-02 & 3.253e-02 & 1.01 \\
        & $1e-08$ & 5.665e-01 & 2.749e-01 & 1.334e-01 & 6.553e-02 & 3.253e-02 & 1.01 \\
        \midrule
        \multirow{2}{*}{$1e+04$} & $1e-06$ & 5.423e-01 & 2.688e-01 & 1.324e-01 & 6.538e-02 & 3.251e-02 & 1.01 \\
        & $1e-08$ & 5.423e-01 & 2.688e-01 & 1.324e-01 & 6.538e-02 & 3.251e-02 & 1.01 \\
        \midrule
        \multirow{2}{*}{$1e+08$} & $1e-06$ & 5.423e-01 & 2.688e-01 & 1.324e-01 & 6.538e-02 & 3.251e-02 & 1.01 \\
        & $1e-08$ & 5.423e-01 & 2.688e-01 & 1.324e-01 & 6.538e-02 & 3.251e-02 & 1.01 \\
        \bottomrule
    \end{tabular}
\end{table}

\begin{table}[htbp]
    \centering
    \caption{Relative error $\frac{\lVert \matr{\sigma}_h-\matr{\sigma}_0 \rVert}{\lVert \matr{\sigma}_0 \rVert}$ and convergence rate in Example~\ref{sec:example-2}.}
    \label{tab:6}
    \footnotesize
    \begin{tabular}{ccccccc c}
        \toprule
        & & \multicolumn{5}{c}{$h$} & \\
        \cmidrule(lr){3-7}
        $\lambda$ & $\iota$ & $1/8$ & $1/16$ & $1/32$ & $1/64$ & $1/128$ & Rate \\
        \midrule
        \multirow{2}{*}{$1e+00$} & $1e-06$ & 1.017e-01 & 3.159e-02 & 8.991e-03 & 2.410e-03 & 6.216e-04 & 1.95 \\
        & $1e-08$ & 1.017e-01 & 3.159e-02 & 8.991e-03 & 2.410e-03 & 6.216e-04 & 1.95 \\
        \midrule
        \multirow{2}{*}{$1e+04$} & $1e-06$ & 1.032e-01 & 3.218e-02 & 9.149e-03 & 2.435e-03 & 6.247e-04 & 1.96 \\
        & $1e-08$ & 1.032e-01 & 3.218e-02 & 9.149e-03 & 2.435e-03 & 6.247e-04 & 1.96 \\
        \midrule
        \multirow{2}{*}{$1e+08$} & $1e-06$ & 1.032e-01 & 3.218e-02 & 9.149e-03 & 2.435e-03 & 6.247e-04 & 1.96 \\
        & $1e-08$ & 1.032e-01 & 3.218e-02 & 9.149e-03 & 2.435e-03 & 6.247e-04 & 1.96 \\
        \bottomrule
    \end{tabular}
\end{table}

\section{Conclusion} \label{sec:8}

We have proposed a hybridized staggered discontinuous Galerkin--mixed finite element method for the Aifantis strain gradient elasticity model. Starting from a first-order reformulation that introduces the total stress together with the scaled Cauchy and hyper stresses as auxiliary unknowns, we discretize the displacement and the total stress by staggered discontinuous Galerkin spaces and the two scaled stresses by Raviart--Thomas pairs, with symmetry imposed strongly on all three stresses. We further developed a hybridized version of the scheme, in which local static condensation leaves only the two multipliers on the mesh skeleton as globally coupled unknowns; a modification suitable for the incompressible limit is given in \Cref{sm:incompressible}.

On the theoretical side, we established an inf-sup condition on the symmetric subspace and a discrete Korn inequality, from which the stability of the scheme follows. The error analysis gives first an optimal-order estimate that tracks the dependence on the regularity of $\vect{u}$ and on $\lambda$ and $\iota$ explicitly, and then derives an estimate uniform in these two parameters. For the hybridized scheme, inf-sup conditions are developed for the multipliers, from which their stability follows. The numerical experiments confirm the predicted rates and the robustness in $\lambda$ and $\iota$, both for a smooth solution and in the presence of strong boundary layers.

The present work is restricted to two space dimensions and to the homogeneous doubly clamped boundary condition. The extension to three dimensions and to general boundary conditions is left to future work.

\appendix
\section{Incompressible Modification of the Hybridized Scheme} \label{sm:incompressible}

As noted in \cref{rmk:incompressible-limit}, the normalization $\int_{\Omega} \tr \matr{\sigma}_h = 0$ is needed to ensure well-posedness in the incompressible limit $\lambda \to \infty$. However, this constraint is global and obstructs the local elimination of $\matr{\sigma}_h$ in the hybridized scheme. We therefore modify that scheme as follows.

We localize the normalization to each dual element by introducing the subspace on which the trace has vanishing mean over every dual element,
\begin{align*}
    \widetilde{\Sigma}^s_{h, 0} \defeq \{ \widetilde{\matr{\tau}}_h \in \widetilde{\Sigma}^s_h : \int_D \tr \widetilde{\matr{\tau}}_h = 0,\ \forall D \in \mathcal{T}^{dl}_h \},
\end{align*}
together with the piecewise-constant space that carries these elementwise means,
\begin{align*}
    \overline{\Sigma}_h \defeq \{ \overline{\tau} \in L^2 (\Omega) : \overline{\tau}|_D \in \mathcal{P}_0 (D),\ \forall D \in \mathcal{T}^{dl}_h;\ \int_{\Omega} \overline{\tau} = 0 \}.
\end{align*}
Every $\matr{\tau}_h \in \widetilde{\Sigma}^s_h$ with $\int_{\Omega} \tr \matr{\tau}_h = 0$ then admits the decomposition $\matr{\tau}_h = \matr{\tau}^0_h + \overline{\tau}_h \matr{I}$ with $\matr{\tau}^0_h \in \widetilde{\Sigma}^s_{h, 0}$ and $\overline{\tau}_h \in \overline{\Sigma}_h$, and likewise for $\matr{\sigma}_h$. Substituting these decompositions into \eqref{eq:hybrid-scheme-1}--\eqref{eq:hybrid-scheme-6} affects four of the six equations, which we treat in turn.
\begin{itemize}
    \item \emph{Equation \eqref{eq:hybrid-scheme-1}.} By the definition \eqref{eq:A^1/2} of $\mathcal{A}^{1/2}$, as $\lambda \to \infty$,
    \begin{align*}
        \mathcal{A}^{\frac{1}{2}} \matr{\tau}_h = \frac{1}{\sqrt{2\mu}} \dev \matr{\tau}_h,
    \end{align*}
    so that $(\mathcal{A}^{1/2} \widetilde{\matr{\sigma}}_h, \matr{\tau}_h) = (\mathcal{A}^{1/2} \widetilde{\matr{\sigma}}_h, \matr{\tau}^0_h)$. By the definition \eqref{eq:b*_h} of $b^*_h$,
    \begin{align*}
        &b^*_h (\vect{u}_h, \matr{\tau}_h) + ( \widehat{\vect{u}}_h, \jump{\matr{\tau}_h \cdot \vect{n}} )_{\mathcal{F}^{dl}_h} \\
        &\qquad = b^*_h (\vect{u}_h, \matr{\tau}^0_h) + ( \widehat{\vect{u}}_h, \jump{\matr{\tau}^0_h \cdot \vect{n}} )_{\mathcal{F}^{dl}_h} + (\vect{u}_h \cdot \vect{n}, \jump{\overline{\tau}_h})_{\mathcal{F}^{pr}_h} + (\widehat{\vect{u}}_h \cdot \vect{n}, \jump{\overline{\tau}_h})_{\mathcal{F}^{dl}_h}.
    \end{align*}

    \item \emph{Equation \eqref{eq:hybrid-scheme-3}.} For the same reason,
    \begin{align*}
        (\mathcal{A}^{\frac{1}{2}} \matr{\sigma}_h, \widetilde{\matr{\tau}}_h) = (\mathcal{A}^{\frac{1}{2}} \matr{\sigma}^0_h, \widetilde{\matr{\tau}}_h).
    \end{align*}

    \item \emph{Equation \eqref{eq:hybrid-scheme-4}.} Integrating by parts in the definition \eqref{eq:b_h-hybrid} of $b_h$ gives
    \begin{align*}
        b_h (\matr{\tau}_h, \vect{v}_h) = - (\divg \matr{\tau}_h, \vect{v}_h)_{\mathcal{T}_h} + (\jump{\matr{\tau}_h \cdot \vect{n}}, \vect{v}_h)_{\mathcal{F}^{pr}_h},
    \end{align*}
    whence
    \begin{align*}
        b_h (\matr{\sigma}_h, \vect{v}_h) = b_h (\matr{\sigma}^0_h, \vect{v}_h) + (\jump{\overline{\sigma}_h}, \vect{v}_h \cdot \vect{n})_{\mathcal{F}^{pr}_h}.
    \end{align*}

    \item \emph{Equation \eqref{eq:hybrid-scheme-5}.} The jump term splits as
    \begin{align*}
        ( \jump{\matr{\sigma}_h \cdot \vect{n}}, \widehat{\vect{v}}_h )_{\mathcal{F}^{dl}_h}
        = ( \jump{\matr{\sigma}^0_h \cdot \vect{n}}, \widehat{\vect{v}}_h )_{\mathcal{F}^{dl}_h} + (\jump{\overline{\sigma}_h}, \widehat{\vect{v}}_h \cdot \vect{n})_{\mathcal{F}^{dl}_h}.
    \end{align*}
\end{itemize}

For brevity we continue to write $\matr{\sigma}_h$ for the component $\matr{\sigma}^0_h$. The modified scheme then reads as follows.

\paragraph{Incompressible hybridized scheme.} Find $(\vect{u}_h, \matr{\sigma}_h, \widetilde{\matr{\sigma}}_h, \tens{\chi}_h, \overline{\sigma}_h, \widehat{\vect{u}}_h, \widehat{\matr{\sigma}}_h) \in U_{h, 0} \times \widetilde{\Sigma}^s_{h, 0} \times \widetilde{\Sigma}^s_h \times \widetilde{X}^s_h \times \overline{\Sigma}_h \times \widehat{U}_h \times \widehat{\Sigma}^s_h$ such that
\begin{align*}
    (\mathcal{A}^{\frac{1}{2}} \widetilde{\matr{\sigma}}_h, \matr{\tau}^0_h) - b^*_h (\vect{u}_h, \matr{\tau}^0_h) - ( \widehat{\vect{u}}_h, \jump{\matr{\tau}^0_h \cdot \vect{n}} )_{\mathcal{F}^{dl}_h} &= 0, & &\forall \matr{\tau}^0_h \in \widetilde{\Sigma}^s_{h, 0}, \\
    (\tens{\chi}_h, \tens{\psi}_h) + \iota (\widetilde{\matr{\sigma}}_h, \divg \tens{\psi}_h)_{\mathcal{T}^{dl}_h} - ( \widehat{\matr{\sigma}}_h, \jump{\tens{\psi}_h \cdot \vect{n}} )_{\mathcal{F}^{dl}_h} &= 0, & &\forall \tens{\psi}_h \in \widetilde{X}^s_h, \\
    (\widetilde{\matr{\sigma}}_h, \widetilde{\matr{\tau}}_h) - (\mathcal{A}^{\frac{1}{2}} \matr{\sigma}_h, \widetilde{\matr{\tau}}_h) - \iota (\divg \tens{\chi}_h, \widetilde{\matr{\tau}}_h)_{\mathcal{T}^{dl}_h} &= 0, & &\forall \widetilde{\matr{\tau}}_h \in \widetilde{\Sigma}^s_h, \\
    b_h (\matr{\sigma}_h, \vect{v}_h) + (\jump{\overline{\sigma}_h}, \vect{v}_h \cdot \vect{n})_{\mathcal{F}^{pr}_h} &= (\vect{f}, \vect{v}_h), & &\forall \vect{v}_h \in U_{h, 0}, \\
    ( \jump{\matr{\sigma}_h \cdot \vect{n}}, \widehat{\vect{v}}_h )_{\mathcal{F}^{dl}_h} + (\jump{\overline{\sigma}_h}, \widehat{\vect{v}}_h \cdot \vect{n})_{\mathcal{F}^{dl}_h} &= 0, & &\forall \widehat{\vect{v}}_h \in \widehat{U}_h, \\
    ( \jump{\tens{\chi}_h \cdot \vect{n}}, \widehat{\matr{\tau}}_h )_{\mathcal{F}^{dl}_h} &= 0, & &\forall \widehat{\matr{\tau}}_h \in \widehat{\Sigma}^s_h, \\
    - (\vect{u}_h \cdot \vect{n}, \jump{\overline{\tau}_h})_{\mathcal{F}^{pr}_h} - (\widehat{\vect{u}}_h \cdot \vect{n}, \jump{\overline{\tau}_h})_{\mathcal{F}^{dl}_h} &= 0, & &\forall \overline{\tau}_h \in \overline{\Sigma}_h.
\end{align*}

Since $\overline{\sigma}_h$ is globally coupled, it joins $(\widehat{\vect{u}}_h, \widehat{\matr{\sigma}}_h)$ on the mesh skeleton, and the local and global problems become the following.

\paragraph{Local problem.}
Given $(\overline{\sigma}_h, \widehat{\vect{u}}_h, \widehat{\matr{\sigma}}_h)$, on each $D \in \mathcal{T}^{dl}_h$ find $(\vect{u}_h, \matr{\sigma}_h, \widetilde{\matr{\sigma}}_h, \tens{\chi}_h) \in U_{h, 0} (D) \times \widetilde{\Sigma}^s_{h, 0} (D) \times \widetilde{\Sigma}^s_h (D) \times \widetilde{X}^s_h (D)$ such that
\begin{align*}
    (\mathcal{A}^{\frac{1}{2}} \widetilde{\matr{\sigma}}_h, \matr{\tau}_h)_D + (\vect{u}_h, \divg \matr{\tau}_h)_{\mathcal{T}_D} - (\vect{u}_h, \jump{\matr{\tau}_h \cdot \vect{n}})_{e_D} &= (\widehat{\vect{u}}_h, \matr{\tau}_h \cdot \vec{\vect{n}})_{\partial D}, \\
    (\tens{\chi}_h, \tens{\psi}_h)_D + \iota (\widetilde{\matr{\sigma}}_h, \divg \tens{\psi}_h)_D &= (\widehat{\matr{\sigma}}_h, \tens{\psi}_h \cdot \vec{\vect{n}})_{\partial D}, \\
    (\widetilde{\matr{\sigma}}_h, \widetilde{\matr{\tau}}_h)_D - (\mathcal{A}^{\frac{1}{2}} \matr{\sigma}_h, \widetilde{\matr{\tau}}_h)_D - \iota (\divg \tens{\chi}_h, \widetilde{\matr{\tau}}_h)_D &= 0, \\
    (\matr{\sigma}_h, \grad \vect{v}_h)_{\mathcal{T}_D} - (\matr{\sigma}_h \cdot \vec{\vect{n}}, \vect{v}_h)_{\partial D} &= (\vect{f}, \vect{v}_h)_D - (\jump{\overline{\sigma}_h}, \vect{v}_h \cdot \vect{n})_{e_D},
\end{align*}
for all $(\matr{\tau}_h, \tens{\psi}_h, \widetilde{\matr{\tau}}_h, \vect{v}_h) \in \widetilde{\Sigma}^s_{h, 0} (D) \times \widetilde{X}^s_h (D) \times \widetilde{\Sigma}^s_h (D) \times U_{h, 0} (D)$.

\paragraph{Global problem.} Find $(\overline{\sigma}_h, \widehat{\vect{u}}_h, \widehat{\matr{\sigma}}_h) \in \overline{\Sigma}_h \times \widehat{U}_h \times \widehat{\Sigma}^s_h$ such that
\begin{align*}
    (\jump{\matr{\sigma}^{(\overline{\sigma}_h, \widehat{\vect{u}}_h, \widehat{\matr{\sigma}}_h)}_h \cdot \vect{n}}, \widehat{\vect{v}}_h)_{\mathcal{F}^{dl}_h} + (\jump{\overline{\sigma}_h}, \widehat{\vect{v}}_h \cdot \vect{n})_{\mathcal{F}^{pr}_h} &= 0, & &\forall \widehat{\vect{v}}_h \in \widehat{U}_h, \\
    (\jump{\tens{\chi}^{(\overline{\sigma}_h, \widehat{\vect{u}}_h, \widehat{\matr{\sigma}}_h)}_h \cdot \vect{n}}, \widehat{\matr{\tau}}_h)_{\mathcal{F}^{dl}_h} &= 0, & &\forall \widehat{\matr{\tau}}_h \in \widehat{\Sigma}^s_h, \\
    - (\vect{u}^{(\overline{\sigma}_h, \widehat{\vect{u}}_h, \widehat{\matr{\sigma}}_h)}_h \cdot \vect{n}, \jump{\overline{\tau}_h})_{\mathcal{F}^{pr}_h} - (\widehat{\vect{u}}_h \cdot \vect{n}, \jump{\overline{\tau}_h})_{\mathcal{F}^{dl}_h} &= 0, & &\forall \overline{\tau}_h \in \overline{\Sigma}_h.
\end{align*}

\bibliographystyle{unsrt}
\bibliography{reference}
\end{document}